\documentclass[11pt,letterpaper]{article}
\usepackage{amsmath,latexsym,amssymb,ifthen}
\usepackage{mathtools}
\usepackage{bbm}
\mathtoolsset{showonlyrefs}
\renewcommand{\emph}[1]{\textit{#1}}
\usepackage{enumerate}
\usepackage{color}\usepackage{graphicx}

\definecolor{brown}{cmyk}{0, 0.72, 1, 0.45}
\definecolor{grey}{gray}{0.5}

\newcommand{\ignore}[1]{}

\newboolean{draft}
\setboolean{draft}{true}
\def\cA{{\cal A}}
\def\cB{{\cal B}}

\def\cT{\mathcal{T}}

\def\pee{{\mathcal P}}
\newcommand{\set}[1]{\left\{#1\right\}}
\def\cP{\mathcal{P}}

\newcommand{\gap}[1]{\mbox{\hspace{#1 in}}}

\newcommand{\proofstart}{{\noindent \bf Proof\hspace{2em}}}
\newcommand{\proofend}{\hspace*{\fill}\mbox{$\Box$}\\ \medskip}

\def\ii_(#1,#2){i_{#1}^{#2}}

\renewcommand{\S}[1]{S^{(#1)}}

\newcommand{\hS}[1]{\hat{S}^{(#1)}}

\def\LL{\Lambda}

\def\a{\alpha}
\def\b{\beta}
\def\d{\delta}

\def\e{\varepsilon}
\def\f{\phi}

\def\g{\gamma}
\def\G{\Gamma}

\def\z{\zeta}

\def\th{\theta}

\def\l{\lambda}
\def\m{\mu}
\def\n{\nu}
\def\p{\pi}

\def\r{\rho}

\def\s{\sigma}
\def\t{\tau}

\newcommand{\rdup}[1]{\left\lceil #1 \right\rceil}
\newcommand{\rdown}[1]{\mbox{$\left\lfloor #1 \right\rfloor$}}

\def\cE{\mathcal{E}}

\def\prob{\Pr}

\newcommand{\brac}[1]{\left( #1 \right)}

\newcommand{\expect}{\operatorname{\bf E}}
\def\E{\expect}
\renewcommand{\Pr}{\operatorname{\bf Pr}}
\newcommand\bfrac[2]{\left(\frac{#1}{#2}\right)}

\newtheorem{theorem}{Theorem}
\newtheorem{lemma}{Lemma}
\newtheorem{corollary}[lemma]{Corollary}

\newtheorem{remthm}[lemma]{Remark}

\newenvironment{remark}{\begin{remthm}\it }{\end{remthm}}%
\newcounter{thmtemp}

\newenvironment{proof}{\proofstart}{\proofend}
\newcommand{\nospace}[1]{}

\def\path{\operatorname{PATH}}

\newcommand{\eee}{\mathbb{E}}

\newcommand{\Bin}{\ensuremath{\operatorname{Bin}}}

\newcommand{\beq}[1]{\begin{equation}\label{#1}}
\def\eeq{\end{equation}}

\def\cC{{\cal C}}
\begin{document}

\title{Elegantly colored paths and cycles in edge colored random graphs}
\date{\today}

\author{Lisa Espig
\thanks{Department of Mathematical Sciences,
Carnegie Mellon University, Pittsburgh PA15213,
Research supported in part by NSF grant DMS-6721878,
\hbox{e-mail}~{\small\texttt{lespig@andrew.cmu.edu}}}
\and
Alan Frieze
\thanks{Department of Mathematical Sciences,
Carnegie Mellon University, Pittsburgh PA15213,
Research supported in part by NSF grant DMS-6721878,
\hbox{e-mail}~{\small\texttt{alan@random.math.cmu.edu}}}
\and
Michael Krivelevich\thanks{School of Mathematical Sciences,
Raymond and Beverly Sackler Faculty of Exact Sciences, Tel Aviv
University, Tel Aviv, 6997801, Israel,
Research supported in part by a USA-Israel BSF grant and by a grant
from the Israel Science Foundation.
\hbox{e-mail}~{\small\texttt{krivelev@post.tau.ac.il}}}}

\maketitle

\markboth{}{}
\thispagestyle{empty}

\begin{abstract}
We first consider the following problem. We
are given a fixed perfect matching $M$ of $[n]$ and we
add random edges one at a time until there is a Hamilton
cycle containing $M$. We show that w.h.p. the hitting time
for this event is the same as that for the first time
there are no isolated vertices in the graph induced by the random edges.
We then use this result for the following problem.
We generate random edges and randomly color them black or white.
A path/cycle is said to be \emph{zebraic} if the colors alternate
along the path. We show that w.h.p. the hitting time
for a zebraic Hamilton cycle
coincides with every vertex meeting at least one edge of each color.
We then consider some related problems and (partially) extend our results to multiple colors. We also briefly consider directed versions.
\end{abstract}

\section{Introduction}
This paper studies the existence of nicely structured objects in (randomly) colored
random graphs. Our basic interest will be in what we call \emph{zebraic} paths and
cycles. We assume that the edges of a graph $G$ have been colored black or white. A path
or cycle will be called \emph{zebraic} if the edges alternate in color along the path. We view this as a variation on the usual theme of \emph{rainbow} paths and cycles that have been well-studied. Rainbow Hamilton cycles in edge colored complete graphs were first studied in Erd\H{o}s, Ne\v{s}et\v{r}il and R\"odl \cite{ENR}. Colorings were constrained by the number of times, $k$, that an individual color could be used. Such a coloring is called $k$-bounded. They showed that allowing $k$ to be any constant, there was always a rainbow Hamilton cycle, provided that the number of vertices $n$ was sufficiently large. Hahn and Thomassen \cite{HT} were next to consider this problem and they showed that $k$ could grow as fast as $n^{1/3}$ and there still be a rainbow Hamilton cycle and conjectured that the growth rate of $k$ could in fact be linear. In an unpublished work R\"{o}dl and Winkler \cite{W} in 1984 improved this to $n^{1/2}$. Frieze and Reed \cite{FR} improved this to $k=O(n/\log n)$ and finally Albert, Frieze and Reed \cite{AFR} (and Rue) improved the upper bound on $k$ to $n/64$. In another line of research, Cooper and Frieze \cite{CF} discussed the existence of rainbow Hamilton cycles in the random graph $G^{(q)}_{n,p}$ which consists of the random graph $G_{n,p}$ where each edge is independently and randomly given one of $q$ colors. Here and elsewhere, we use ``chosen randomly'' to signify ``chosen uniformly at random''. They showed that if $p\geq \frac{21\log n}{n}$ and $q\geq 21n$ then with high probability (w.h.p.), i.e. probability $1-o(1)$, there is a rainbow colored Hamilton cycle. Frieze and Loh \cite{FL} improved this to $p\geq \frac{(1+o(1))\log n}{n}$ and $q\geq n+o(n)$. Ferber and Krivelevich \cite{FK15} improved it further to $p=\frac{\log n+\log\log n+\omega(n)}{n}$ and $q\geq n+o(n)$. Bal and Frieze \cite{BF15} considered the case $q=n$ and showed that $p\geq \frac{K\log n}{n}$ suffices for large enough $K$. Ferber and Krivelevich \cite{FK15} proved that if $p\gg \frac{\log n}{n}$ and $q=Cn$ colors are used, then w.h.p. $G_{n,p}$ contains $(1-o(1))np/2$ edge-disjoint rainbow Hamilton cycles, for $C$ large enough.

In this paper we study the existence of other colorings of paths and cycles. Our first result does not at first sight fit into this framework. Let $n$ be even and let $M_0$ be an arbitrary perfect matching of the complete graph $K_n$. Now consider the random graph process $\{G_m\}=\{([n],E_m)\}$ where $E_m=\set{e_1,e_2,\ldots,e_m}$ is obtained from $E_{m-1}$ by adding a random edge $e_m\notin E_{m-1}$, for $m=0,1,\ldots, N=\binom{n}{2}$.

Let
$$\t_1=\min\set{m:\d(G_m)\geq 1}\,,$$
where $\d$ denotes minimum degree. Then let
$$\t_H=\min\set{m:G_m\cup M_0\text{ contains a Hamilton cycle }H\supseteq M_0}.$$
\begin{theorem}\label{th1}
$\t_1=\t_H$ w.h.p.
\end{theorem}
\begin{remark}\label{remmm}
In actual fact there are two slightly different versions. One where we insist that $M_0\cap E_m=\emptyset$ and one where $E_m$ is chosen completely independently of $M_0$. Our proof of the theorem covers both cases. We will first give a proof under the assumption that $E_m$ is chosen independently and then in Remark \ref{disjoint} see how to obtain the other case.
\end{remark}

\medskip

We note that Robinson and Wormald \cite{RW} considered a similar problem with respect to random regular graphs. They showed that one can choose $o(n^{1/2})$ edges at random,
orient them and then w.h.p. there will be a Hamilton cycle containing these edges and following the orientations.

Theorem \ref{th1} has an easy corollary that fits our initial description. Let $\{G^{(r)}_m\}$ be an $r$-colored version of the graph process. This means that $G^{(r)}_m$ is obtained from $G^{(r)}_{m-1}$ by adding a random edge and then giving it a random color from $[r]$. Let $E_{m,i}$ denote the edges of color $i$ in $\{G^{(r)}_m\}$ for $i=1,2,\ldots,r$. When $r=2$ denote the colors by $black$ and $white$ and let $E_{m,b}=E_{m,1},E_{m,w}=E_{m,2}$. Then let $G^{(b)}_m$ be the subgraph of $G^{(2)}_m$ induced by the black edges and let $G^{(w)}_m$ induced by the white edges. Let 
$$\t_{1,1}=\min\set{m:\d(G^{(b)}_m),\d(G^{(w)}_m)\geq 1}\,,$$
and let
$$\t_{ZH}=\min\set{m:G_m^{(2)}\text{ contains a zebraic Hamilton cycle}}.$$
\begin{corollary}\label{cor1}
$\t_{1,1}=\t_{ZH}$ w.h.p.
\end{corollary}
Our next result is a zebraic analogue of \emph{rainbow connection}. For a connected graph $G$, its rainbow connection $rc(G)$, is the minimum number $r$ of colors needed for the following to hold: The edges of $G$ can be $r$-colored so that every pair of vertices is connected by a rainbow path, i.e. a path in which no color is repeated. Recently, there has been interest in estimating this parameter for various classes of graph, including random graphs (see, e.g., \cite{DFT, FT12, HR12, KKS15}). By analogy, we say that a connected graph with a two-coloring of its edges is \emph{zebraicly connected} if there is a zebraic path joining every pair of vertices.
\begin{theorem}\label{th4}
At time $\t_1$, $G_{\t_1}$ with a random black-white coloring of its edges is zebraicly connected, w.h.p.
\end{theorem}
We consider now how we can extend our results to more than two colors. Suppose we have $r$ colors $[r]$ and that $r\mid n$. We would like to consider the existence of Hamilton cycles where the $i$th edge has color $(i\mod r)+1$. Call such a cycle \emph{ $r$-zebraic}. Our result for this case is not as tight as for the case of two colors. We are not able to prove a hitting time version. We will instead satisfy ourselves with a result for $G_{n,p}^{(r)}$. Let
$$p_r=\frac{r}{\a_r}\frac{\log n}{n}$$
where
$$\a_r=\rdup{\frac{r}{2}}.
$$
\begin{theorem}\label{th5}
Let $\e>0$ be an arbitrary positive constant and suppose that $r\geq 2$.
$$\lim_{n\to\infty}\Pr(G_{n,p}^{(r)}\text{ contains an $r$-zebraic Hamilton
cycle})=\begin{cases}
         0&p\leq (1-\e)p_r\\1&p\geq (1+\e)p_r
        \end{cases}.
$$
\end{theorem}
The proofs of Theorems \ref{th1}--\ref{th5} will be given in Sections \ref{th1proof}--\ref{th5proof}.
\subsection{Directed Versions}\label{directed}
There are some very natural directed versions of these results. With respect to Theorem \ref{th1} one can consider the directed graph process where the edges of the complete digraph $\vec{K}_n$ are randomly ordered as $e_1,e_2,\ldots,e_{n(n-1)}$. We can then consider a sequence of digraphs $D_m=([n],\set{e_1,e_2,\ldots,e_m}),m\geq 1$ and consider hitting times for various properties. For example, suppose in addition one is given a perfect matching $M=\set{f_1,f_2,\ldots,f_{n/2}}$ together with an orientation of each edge in $M$. One can ask for the likely hitting time for the existence of a directed Hamilton cycle that contains $M$ and respects the given orientation. Assume w.l.o.g. that $f_i=(2i-1,2i)$ for $i=1,2,\ldots,n/2$, so that $f_i$ is oriented from $2i-1$ to $2i$. Let $\vec{\t}_{H}$ be the hitting time for the existence of such a cycle. Let $\vec{\t}_1$ be the hitting time for each $1\leq i\leq n/2$ to have an in-neighbor in $n/2+1,n/2+2,\ldots,n$ and for each $n/2+1,n/2+2,\ldots,n$ to have an out-neighbor in $1\leq i\leq n/2$. Clearly $\vec{\t}_{H}\geq \vec{\t}_1$.
\begin{theorem}\label{th1d}
$\vec{\t}_1=\vec{\t}_H$ w.h.p.
\end{theorem}
Our other results will have directed analogs too. Suppose then that $D_{n,p}^{(r)},m\geq 1$ is an $r$-colored version of the directed graph $D_{n,p}$. A directed $r$-zebraic Hamilton cycle is the directed analog what we see in Theorem \ref{th5}. Then we have
\begin{theorem}\label{th5d}
Let $\e>0$ be an arbitrary positive constant and suppose that $r\geq 2$.
$$\lim_{n\to\infty}\Pr(D_{n,p}^{(r)}\text{ contains an $r$-zebraic directed  Hamilton
cycle})=\begin{cases}
         0&p\leq (1-\e)p_r\\1&p\geq (1+\e)p_r
        \end{cases}.
$$
\end{theorem}
Notice that we do not claim a hitting time version for the case $r=2$. It is unclear what the simple necesary condition should be. We discuss this further in Section \ref{diranalog}.

There is a notion of directed zebraic connection when we 2-color a digraph and ask for a directed zebraic path from any vertex to any other vertex. Let $\vec{\t}_{1,1}$ be the hitting time for $D_m$ to have i-degree and out-degree at least one. 
\begin{theorem}\label{th4d}
At time $\vec{\t}_{1,1}$, $D_{\vec{\t}_{1,1}}$ with a random black-white coloring of its edges is directed zebraicly connected, w.h.p.
\end{theorem}
We will briefly discuss the proofs of these directed analogs in Section \ref{diranalog}.
\section{Notation}\label{notation}
All logarithms will have base $e$ unless explicitly stated otherwise.

For a graph $G=(V,E)$ and $S,T\subseteq V$ we let $e_G(S)$ denote the number of edges contained in $S$, $e_G(S,T)$ denote the number of edges with one end in $S$ and the other in $T$. Let $e_G(S)=e_G(S,S)$ and let $N_G(S)$ denote the set of neighbors of $S$ that are not in $S$.

We next list certain values and notation that we will use throughout our proofs. They are here for easy reference. The reader is encouraged to skip reading this section and to just refer back as necessary.
\begin{align*}
&t_0=\frac{n}{2}(\log n -2\log\log n)\text{ and }
t_1=\frac{n}{2}(\log n +2\log\log n)\\
&t_2=\frac{t_0}{10}\text{ and }t_3=\frac{t_0}{5}\text{ and }t_4=\frac{9t_0}{10}.\\
&\z_i=t_i-t_{i-1}\text{ for }i=3,4.\\
&p_i=\frac{t_i}{\binom{n}{2}},\,i=0,1,2.\\
&n_0=\frac{n}{\log^2n}\text{ and }n_0'=\frac{n_0}{\log^4n}\text{ and }
n_1=\frac{n}{10\log n}.\\
&n_b=\frac{n\log\log\log n}{\log\log n}\text{ and }n_c =\frac{200n}{\log n}.\\
&L_0=\frac{\log n}{100}\text{ and }L_1=\frac{\log n}{\log\log n}.\\
&\ell_0=\frac{\log n}{200}\text{ and }\ell_1=\frac{2\log n}{3\log\log n}\text{ and }
\n_L=\ell_0^{\ell_1}=n^{2/3+o(1)}.
\end{align*}
The following graphs and sets of vertices are used.
\begin{align*}
&\Psi_0=G_{t_2}\setminus M_0=([n],E_{t_2}\setminus M_0).\\
&V_0=\set{v\in [n]:d_{\Psi_0}(v)\leq L_0}.\\
&\Psi_1=\Psi_0\cup\set{e\in E_{\t_1}\setminus  E_{t_2}:e\cap V_0\neq\emptyset}.\\
&V_\l=\set{v\in [n]:\;v\text{ is large}}.\\
&V_\s=[n]\setminus V_{\l}.\\
&E_B=\set{e\in E_{t_4}\setminus E_{t_3}: e\cap V_0=\emptyset}.\\
&V_\t=\set{v\in [n]\setminus V_0:\deg_{E_B}(v)\leq L_0}.
\end{align*}
The definition of ``large'' depends on which theorem we are proving.

Sometimes in what follows we will treat certain values as integer, when they should really be rounded up or down. We do this for conveneience and claim that rounding either way will not affect the validity of what is claimed.
\section{Probabilistic Inequalities}
We will need standard estimates on the tails of various random variables.

{\bf Chernoff Bounds:} Let $B(n,p)$ denote the binomial random variable where $n$ is the number of trials and $p$ is the probability of success.
\begin{align}
&\Pr(|B(n,p)-np|\geq \e np)\leq 2e^{-\e^2np/3}\qquad\text{for }0\leq\e\leq 1.\label{chern1}\\
&\Pr(B(n,p)\geq anp)\leq \bfrac{e}{a}^{anp}\qquad\text{for }a>0.\label{chern2}
\end{align}
For proofs, see the appendix of Alon and Spencer \cite{AS}.

{\bf McDiarmid's Inequality:} Let $Z=Z(Y_1,Y_2,\ldots,Y_n)$ be a random variable where $Y_1,Y_2,\ldots,Y_n$ are independent for $i=1,2,\ldots,n$. Suppose that
$$|Z(Y_1,\ldots,Y_{i-1},Y_i,Y_{i+1},\ldots,Y_n)-Z(Y_1,\ldots,Y_{i-1},\widehat{Y}_i,Y_{i+1},\ldots,Y_n)|\leq c_i$$
for all $Y_1,Y_2,\ldots,Y_n,\widehat{Y}_i$ and $1\leq i\leq n$. Then
\beq{mcd}
\Pr(|Z-\E(Z)|\geq t)\leq 2\exp\set{-\frac{t^2}{2(c_1^2+c_2^2+\cdots+c_n^2)}}.
\eeq
For a proof see for example \cite{FK15B}(Lemma 21.16) or \cite{JLR}(Remark 2.28).
\section{Proof of Theorem \ref{th1}}\label{th1proof}
\subsection{Outline of proof}
It is well known (see for example \cite{FK15B}(Theorem 4.2) and \cite{JLR}(Section 5.1)) that w.h.p. we have $t_0\leq \t_1\leq t_1$.

Our strategy for proving Theorem \ref{th1} is broadly in line with the 3-phase algorithm described in \cite{CF1}.
\begin{description}
\item[(a)]
We will take the first $t_3$ edges plus all of the next $\t_1-t_3$ edges incident to vertices that have a low degree in $G_{t_2}$. We argue that w.h.p. this contains a perfect matching $M_1$ that is disjoint from $M_0$. The union of $M_0,M_1$ will then have $O(\log n)$ components w.h.p.
\item[(b)] $M_0\cup M_1$ induces a 2-factor made up of alternating cycles. We then use a selection of about $\z_4$ edges from $E_{t_4}\setminus E_{t_3}$ to make the minimum cycle length $\Omega(n/\log n)$. This selection is carefully designed to avoid dependence issues, as is the case of the selection in (c).
\item[(c)] We then use a large subset of the final $t_2$ edges to create a Hamilton cycle containing $M_0$. This involves a second moment calculation. The edges used to create the cycle here are from $E_{t_0}\setminus E_{t_4}$. It follows that w.h.p. we will have created a Hamilton cycle contained in $G_{\t_1}$.
\end{description}
We are working in a different model to that in \cite{CF1} and there are many more conditioning problems to be overcome. For example, in \cite{CF1}, it is very easy to show that the random digraph $D_{3-in,3-out}$ contains a set of $O(\log n)$ vertex disjoint cycles that contain all vertices. Here we have to build a perfect matching $M_1$ from scratch and to avoid several conditioning problems. The same is true for (b) and (c). The broad strategy is the same, the details are quite different.
\subsection{Phase 1: Building $M_1$}
We begin with $\Psi_0=G_{t_2}\setminus M_0$. Then let $V_0$ denote the set of vertices that have degree at most $L_0=\frac{\log n}{100}$ in $\Psi_0$. Now create $\Psi_1=([n],E_1)$ by adding those edges in $E_{\t_1}\setminus E_{t_2}$ that are incident with $V_0$ and are disjoint from $M_0$. We argue that w.h.p. $\Psi_1$ is a random graph with minimum degree one in which almost all vertices have degree $\Omega(\log n)$. Furthermore, we will show that w.h.p. $\Psi_1$ is an expander, and then it will not be difficult to show that it contains the required perfect matching $M_1$.

Let a vertex be \emph{large} if its degree in $G_{t_1}$ is at
least $L_0$
and \emph{small} otherwise. Let $V_\l$ denote the set of
large vertices and
let $V_\s$ denote the set of small vertices.

The calculations for the next lemma will simplify if we observe
the following: Suppose that $m=Np$. It is known that for any monotone property of graphs
\beq{mon}
\Pr(G_m\in \cP)\leq 3\Pr(G_{n,p}\in \cP).
\eeq
In general we have for not necessarily monotone properties:
\beq{notmon}
\Pr(G_m\in \cP)\leq 3m^{1/2}\Pr(G_{n,p}\in \cP).
\eeq
For proofs of \eqref{mon}, \eqref{notmon} see Bollob\'as \cite{Book}(Theorem 2.2)
or Frieze and Karo\'nski \cite{FK15B}(Lemmas 1.2 and 1.3) or Janson, {\L}uczak and Ruci\'nski \cite{JLR}(Lemma 1.10).

We will have reason to deal with a random sequence of multi-graphs defined as follows: Let $x_1,x_2,\ldots,x_t,\ldots,$ be random sequence where for all $i\geq 0$, $x_{i+1}$ is chosen uniformly at random from $[n]$, independently of $x_1,x_2,\ldots,x_i$. For a positive integer $t$ we let $\G_t$ be the multi-graph with edges $e_1,e_2,\ldots,e_t$ where $e_i=\set{x_{2i-1},x_{2i}}$ for $i\geq 1$. If after removing the loops and repeats of edges from $\G_t$ we have $\t$ edges then the graph we obtain has the same distribution as $G_\t$. Given this, we couple $\G_t$ with $G_\t$ where $\t=\t(\G_t)$ is a random variable.

Let $Z_1$ denote the number of loops and let $Z_2$ denote the number of repeated edges in $\G_{t_2}$. Now $Z_1$ is distributed as $\Bin(t_2,1/n)$ and then the Chernoff bound \eqref{chern2} implies that
\beq{Z1}
\Pr(Z_1\geq \log^2n)\leq e^{-\log^2n}.
\eeq
We are doing more than usual here, because we need probability $O(n^{-0.51})$, rather than just probability $o(1)$. We require this in order to enable us to easily  handle the case where we have to choose edges disjoint from $M_0$, as explained in Remark \ref{disjoint}. There is one exception to this probabilistic requirement. Let $\cT$ be the event that $t_0\leq\t_1\leq t_1$. We do not require that $\Pr(\cT)=1-O(n^{-0.51})$. A probability of $1-o(1)$ will suffice.

Now $Z_2$ is dominated by $\Bin(t_2,t_2/N)$ and then the Chernoff bound \eqref{chern2} implies that
\beq{Z2}
\Pr(Z_2\geq \log^3n)\leq e^{-\log^3n}.
\eeq

The properties in the next lemma will be used to show that w.h.p. $\Psi_1$ is an expander.
\begin{lemma}\label{lem0}
The following holds with probability $1-O(n^{-0.51})$:
\begin{description}
\item[(a)] $|V_0|\leq n^{99/100}$.
\item[(b)] If $x,y\in V_\s$ then the distance between them in $G_{t_1}$
is at least 10.
\item[(c)] If $S\subseteq [n]$ and $|S|\leq n_0=\frac{n}{\log^2n}$ then
$e_{G_{t_1}}(S)\leq 10|S|$.
\item[(d)] If $S\subseteq [n]$ and $|S|=s\in [n_0'=\frac{n_0}{\log^4n},n_1=\frac{n}{10\log n}]$
then $|N_{\Psi_1}(S)|\geq s\log n/25$.
\item[(e)] No cycle of length 4 in $G_{t_1}$ contains a small vertex.
\item[(f)] No vertex of degree one in $G_{\t_1}$ is incident with an edge of $E_{t_1}\cap M_0$.
\item[(g)] The maximum degree in $G_{10n\log n}$ is less than $100\log n$. {\bf Crude but easy to verify.}
\item[(h)] $\t_1\leq 10n\log n$.
\end{description}
\end{lemma}
\proofstart
(a)
Let $V_0'$ denote the set of vertices of degree at most $L_0+1$ in $\G_{t_2}$. Then in our coupling $|V_0|\leq Z_1+2Z_2+|V_0'|$. This is because if $v\in V_0\setminus V_0'$ then it must lie in a loop or a multiple edge. Also, $v\in L_0$ might have degree $L_0+1$ in $G_{t_2}$, but might lose an edge from the deletion of $M_0$ to create $\Psi_1$.

Now, applying \eqref{chern1} with $\e=4/5$ we get
\beq{V0dash}
\prob\brac{v\in V_0'}  \leq \Pr\brac{B\brac{2t_2,\frac{1}{n}}\leq \frac{\log n}{100}+1}\leq n^{-1/99}.
\eeq
It follows, that $\expect(|V_0'|)\leq n^{98/99}$.
We now use inequality \eqref{mcd} to finish the proof. Indeed, changing
one of the $x_i$'s can change $|V_0'|$ by at most one. Hence, for any $u>0$,
$$\Pr(|V_0'|\geq \E(|V_0'|)+u)\leq \exp\set{-\frac{u^2}{4t_2}}.$$
Putting $u=n^{2/3}$ into the above and using \eqref{Z1}, \eqref{Z2} finishes the
proof of (a).

(b)
We do not have room to apply \eqref{notmon} here. We need the inequality
\beq{binom}
\frac{\binom{N-a}{t-b}}{\binom{N}{t}}\leq \bfrac{t}{N}^b\bfrac{N-t}{N-b}^{a-b}
\eeq
for $b\le a\le t\le N$. Verification of \eqref{binom} is straightforward and can be found for example in Chapter 21.1 of \cite{FK15B}. We will now and again use the notation $A\leq_b B$ in place
of $A=O(B)$ when it suits our aesthetic taste. Let $\ell_1=\frac{2\log n}{3\log\log n}$.
\begin{align*}
\Pr(\exists\ x,y) & \leq \sum_{k=2}^{11}
\binom{n }{ k} k! \sum_{\ell_1,\ell_2=0}^{L_0}\binom{n-k}{\ell_1}\binom{n-k}{\ell_2}
\frac{\binom{N-(2n+k-5)}{t_1-(k-1+\ell_1+\ell_2)}}{\binom{N}{t_1}}\\
& \leq_b \sum_{k=2}^{11}n^k \sum_{\ell_1,\ell_2=0}^{L_0}\bfrac{ne}{\ell_1}^{\ell_1}
\bfrac{ne}{\ell_2}^{\ell_2}\bfrac{t_1}{N}^{\ell_1+\ell_2+k-1}
\bfrac{N-t_1}{N-(\ell_1+\ell_2+k-1)}^{2n-(\ell_1+\ell_2-4)} \\
& \leq_b n\sum_{k=2}^{11}\log^{k-1}n \sum_{\ell_1,\ell_2=0}^{L_0}
\bfrac{3\log n}{\ell_1}^{\ell_1}\bfrac{3\log n}{\ell_2}^{\ell_2}n^{-2+o(1)}\\
&=O(n^{-0.51}).
\end{align*}

(c) We can use \eqref{mon} here with $p_1=t_1/N$. If $s = |S|$, then in $G_{n,p_1}$
where $p_1=t_1/N$ and $N=\binom{n}{2}$,
$$\prob(e_{G_{t_1}}(S) > 10|S|) \leq 3\binom{\binom{s}{2}}{10s}p_1^{10s}
\leq 3\brac{\frac{s^2e}{20s}\cdot \frac{\log n+2\log\log n}{n-1}}^{10s}\leq\bfrac{s\log n}{n}^{10s} .$$
So,
$$\Pr(\exists\ S)\leq \sum_{s=10}^{n_0}\binom{n}{s}\bfrac{s\log n}{n}^{10s}
\leq \sum_{s=10}^{n_0}\bfrac{ne}{s}^s\bfrac{s\log n}{n}^{10s}
= \sum_{s=10}^{n_0}\brac{e\bfrac {s}{n}^9\log^{10}n}^s=O(n^{-0.51}).$$
(d) For this we will only use $E_{t_2}\subseteq E(\Psi_1)$. We can use \eqref{mon} here with $p_2=t_2/N$. For $v\in V\setminus S$, $\prob(v\in N_{\Psi_1}(S))\geq 1-(1-p_2)^{s-1} \ge\frac{sp_2}{2}$ for $s\leq n_1$. Here we have $s-1$ in place of $s$ as we need to exclude the edges of $M_0$ in this calculation. So $|N_{\Psi_1}(S)|$ stochastically dominates
$\Bin(n-s, \frac{sp_2}{2})$. Now $(n-s)\frac{sp_2}{2}\sim \frac{s\log n}{20}$ and so using the Chernoff bound \eqref{chern1} with $\e\sim1/5$,
$$\Pr(|N_{\Psi_1}(S)|< s\log n/25) \leq e^{-s\log n/1501}.$$
So,
$$\Pr(\exists\ S)\leq \sum_{s=n_0'}^{n_1}\binom{n}{s}e^{-s\log n/1501}
\leq \sum_{s=n_0'}^{n_1}\brac{\frac{ne}{s}\cdot n^{-1/1501}}^s=O(n^{-0.51}).$$
(e) The expected number of such cycles is bounded by
\begin{align*}
&\binom{n}{4}\frac{3!}{2}\sum_{k=0}^{L_0}4\binom{n-4}{k}
\frac{\binom{N-n-3}{t_1-4-k}}{\binom{N}{t_1}}\\
&\leq n^4\bfrac{t_1}{N}^4\bfrac{N-t_1}{N-4}^{n-1}+n^4\sum_{k=1}^{L_0}\bfrac{ne}{k}^k
\bfrac{t_1}{N}^{k+4}\bfrac{N-t_1}{N-k-4}^{n-k-1}\\
&\leq_b \log^4n \brac{1+\sum_{k=1}^{L_0}\bfrac{e^{1+o(1)}\log n}{k}^k} n^{-1+o(1)}\\
&=O(n^{-0.51}).
\end{align*}

(f) We will first argue that if $V_1$ is the set of vertices of degree at most one in $G_{t_0}$ then
\beq{V1small}
\Pr(|V_1|\geq 2\log^4n)=O(n^{-0.51}).
\eeq
Indeed, fix a set $U\subseteq V$ of size $u$. For $v\in U$, let $d(v,V\setminus U)$ denote the number of edges incident with $v$ and $V\setminus U$. Then, the the probability $U$ is a subset of $V_1$ in $G_{n,p_0}$ is at most 
\begin{align*}
\Pr(d(v,V\setminus U)\le 1, \forall v\in U)&= ((1-p_0)^{n-u}+(n-u)p_0(1-p_0)^{n-u-1})^u\\
&< \bfrac{\log^3n}{n}^u.
\end{align*}
Hence, with $u=\log^4n$ we have 
$$\Pr(|V_1|\geq u)\leq \binom{n}{u}\bfrac{\log^3n}{n}^{u}\leq \brac{\frac{ne}{u}\cdot\frac{\log^3n}{n}}^u=o(n^{-2}).$$
We now apply \eqref{notmon} to prove the result for $G_{t_0}$. 

We now consider adding the final $\max\set{0,\t_1-t_0}$ edges. (We only know that $\Pr(\t_1\geq t_0)=1-o(1)$ and not $1-O(n^{-0.51})$ and so we do not assume that $\t_1\geq t_0$ here.) Let $\cB$ be the event that any of these edges is (i) incident with $V_1$ and (ii) lies in $M_0$. Thus,
$$\Pr(\cB)\leq O(n^{-0.51})+10n\log n\cdot\frac{4\log^4n}{n}\cdot\frac{1}{n}= O(n^{-0.51}).$$
Here the first  $O(n^{-0.51}))$ accounts for the probability that $\t_1-t_0\geq 10n\log n$ or $|V_1|\geq 2\log^4n$. Note that $\Pr(\t_1\geq 10n\log n)=o(n^{-1})$. The proof of this follows from a straigthforward estimate of the expected number of components of size at most $n/2$ at time $10n\log n$, see for example the proof of Theorem 4.1.of \cite{FK15B}. After this, each of the at most $10n\log n$ edges (see (f)) has probability $\frac{4\log^4n}{n}\cdot\frac{1}{n}$ of being in $M_0$ and being incident with $V_1$.

(g) We apply \eqref{mon} with $p=10n\log n/N$ and find that the probability of having a vertex of degree exceeding $100\log n$ is at most
\beq{10logn}
3n\binom{n-1}{100\log n}\bfrac{20\log n}{n-1}^{100\log n}\leq 3n\bfrac{e^{1+o(1)}}{5}^{100\log n}=O(n^{-10}).
\eeq

(h) $\t_1>10n\log n$ implies that $G_{10n\log n}$ contains a component of size at most $n/2$. Thus, with $p$ as in (h),
$$\Pr(\t_1>10n\log n)\leq 3\sum_{k=1}^{n/2}\binom{n}{k}(1-p)^{k(n-k)}\leq 3\sum_{k=1}^{n/2}n^ke^{-5k\log n}=o(1).$$
\proofend
\begin{remark}\label{remT}
Because $\cT$ occurs w.h.p. we have that the statements in Lemma \ref{lem0} hold with probability $1-O(n^{-0.51})$ if we condition on $\cT$ occuring. This follows from $\Pr(A\mid B)\leq \Pr(A)/\Pr(B)$. Indeed, this is also true for any of the events below that are shown to hold with this probability.
\end{remark}
Lemma \ref{lem0} implies the following:
\begin{lemma}
With probability $1-O(n^{-0.51})$,
\beq{eq1}
S\subseteq [n]\text{ and }|S|\leq n/2000\text{ implies }
|N_{\Psi_1}(S)|\geq |S|.
\eeq
\end{lemma}
\proofstart
Assume that the conditions described in Lemma \ref{lem0} hold.
Let $N(S)=N_{\Psi_1}(S)$ and $e(S)=e_{\Psi_1}(S)$.
We first argue that if $S\subseteq V_\l$ and $|S|\leq n/2000$ then
\beq{largeS}
|N(S)|\geq 4|S|.
\eeq
{From} Lemma \ref{lem0}(d), we only have to concern ourselves with $|S|\leq n_0'$
or $|S|\in [n_1,n/2000]$.

If $|S|\leq n_0'$ and $T=N(S)$ then in $\Psi_1$ we have, using Lemma \ref{lem0}(g),(h), and accounting for the edges in $M_0$ being forbidden,
\beq{cof}
e(S\cup T)\geq |S|\brac{\frac{\log n}{200}-1}\text{ and }|S\cup T|\leq |S|\brac{1+100\log n}\leq n_0.
\eeq
It is important to note that to obtain \eqref{cof}
we use the fact that vertices in $V_0\setminus V_\s$
are given all their edges in $\Psi_1$.

Equation \eqref{cof} and Lemma \ref{lem0}(c)
imply that $\frac{|S|\log n}{200}\leq 10|S\cup T|$ and so \eqref{largeS}
holds with room to spare.

If $|S|\in [n_1,n/2000]$ then we choose $S'\subseteq S$ where $|S'|=n_1$ and using
Lemma \ref{lem0}(d), see that
$$|N(S)|\geq |N(S')|-|S|\geq \frac{\log n}{25}\cdot \frac{200|S|}{\log n}-|S|.$$
This yields \eqref{largeS}, again with room to spare.

Now let $S_0=S\cap V_\s$ and $S_1=S\setminus S_0$. Then we have
\beq{N(S)}
|N(S)|\geq |N(S_0)|+|N(S_1)|-|N(S_0)\cap S_1|-|N(S_1)\cap S_0|-|N(S_0)\cap N(S_1)|.
\eeq
But $|N(S_0)|\geq |S_0|$. This follows from (i) $\Psi_1$ has no isolated vertices
(follows from Lemma \ref{lem0}(f)), and (ii) Lemma \ref{lem0}(b) means that $S_0$ is an independent set and no two vertices in $S_0$ have a common neighbor. Equation \eqref{largeS} implies that $|N(S_1)|\geq 4|S_1|$. We next observe that trivially, $|N(S_0)\cap S_1|\leq |S_1|$.
Then we have
$|N(S_1)\cap S_0|\leq |S_1|$, for otherwise some vertex in $S_1$ has two neighbors
in $S_0$, contradicting Lemma \ref{lem0}(b).
Finally, we also have $|N(S_0)\cap N(S_1)|\leq |S_1|$.
If for a vertex in $S_1$ there are
two distinct paths of length two to $S_0$ then we violate one of
the conditions -- Lemma \ref{lem0}(b) or (e).

So, from \eqref{N(S)} we have
$$|N(S)|\geq |S_0|+4|S_1|-|S_1|-|S_1|-|S_1|= |S|.$$
\proofend

Next let $G=(V,E)$ be a graph with an even number of vertices that does not contain a perfect matching. Let $v$ be a vertex not covered by some maximum matching, and suppose that $M$ is a maximum matching that isolates $v$. Let $S_0(v,M)=\set{u\neq v:M\text{ isolates }u}$. If $u\in S_0(v,M)$ and $e=\set{x,y}\in M$ and $f=\set{u,x}\in E$ then {\em flipping} $e,f$ replaces $M$ by $M'=M+f-e$. Here $e$ is {\em flipped-out}. Note that $y\in S_0(v,M')$.

Now fix a maximum matching $M$ that isolates $v$ and let 
$$A(v,M)=\bigcup_{M'}S_0(v,M')$$
where we take the union over $M'$ obtained from $M$ by a sequence of flips.
\begin{lemma}\label{lem4a}
Let $G$ be a graph without a perfect matching and let $M$ be a maximum matching and $v$ be a vertex isolated by $M$. Then $|N_G(A(v,M))|<|A(v,M)|$.
\end{lemma}
\proofstart
Suppose that $x\in N_G(A(v,M))$ and that $f=\set{u,x}\in E$ where $u\in A(v,M)$. Now there exists $y$ such that $e=\set{x,y}\in M$, else $x\in S_0(v,M)\subseteq A(v,M)$. We claim that $y\in A(v,M)$ and this will prove the lemma. Since then, every neighbor of $A(v,M)$ is also a neighbor via an edge of $M$.

Suppose that $y\notin A(v,M)$. Let $M'$ be a maximum matching that (i) isolates $u$ and (ii) is obtainable from $M$ by a sequence of flips. Now $e\in M'$ because if $e$ has been flipped out then either $x$ or $y$ is placed in $A(v,M)$. But then we can do another flip with $M'$, $e$ and the edge $f=\set{u,x}$, placing $y\in A(v,M)$, contradiction.

\proofend

Define 
$$E_A=E_{t_3}\setminus E(\Psi_1)=\set{f_1,f_2,\ldots,f_\r}$$
where we see from Lemma \ref{lem0}(a),(g),(h) that with probability $1-O(n^{-0.51})$ we have
$$\z_3\geq \r \geq \z_3-100n^{99/100}\log n\sim\frac{n\log n}{20}.$$ 
\begin{lemma}\label{lem4}
Given $\Psi_1,V_0,\r$ where $|V_0|\leq n^{99/100}$, we have that $E_A$ is a uniformly random $\r$-subset of
$E_2=\binom{V_1}{2}\setminus E(\Psi_1)$, where $V_1=[n]\setminus V_0$.
\end{lemma}
\proofstart
This follows from the fact that if we remove any $f_i$ and replace it with any other edge from $E_2$ then $V_0$ is unaffected. Thus $E$ and $E-f+g$ are equally likely to be $E_A$, under our conditioning, where $f\in E$ and $g\in E_2\setminus E$. A sequence of such changes shows that any $\r$-subset of $E_2$ is equally likely to be $E_A$.
\proofend

Now consider the sequence of graphs $H_0=\Psi_1,H_1,\ldots,H_\r$ where $H_i$ is obtained from $H_{i-1}$ by adding the edge $f_i$. We claim that if $\m_i$ denotes the size of a largest matching in $H_i$ that is disjoint from $M_0$, then
\beq{match}
\Pr(\m_i= \m_{i-1}+1\mid \m_{i-1}<n/2,\,f_1,
\ldots,f_{i-1},\,(\Psi_1\text{ satisfies \eqref{eq1}}))\geq 10^{-7}.
\eeq
To see this, let $M_{i-1}$ be a matching of size $\m_{i-1}$ in $H_{i-1}$, disjoint from $M_0$, and suppose that $v$ is a vertex not covered by $M_{i-1}$. It follows from \eqref{eq1} and Lemma \ref{lem4a} that if $A_{H_{i-1}}(v)=\set{g_1,g_2,\ldots g_r}$ then $r\geq n/2000$. Now consider the pairs $(g_j,x),\,j=1,\ldots,r,\,x\in A_{H_{j-1}}(g_j)$. There are at least $\binom{n/2000}{2}$ such pairs and if $f_i$ lies in this collection, then $\m_i=\m_{i-1}+1$. Equation \eqref{match} follows from this and Lemma \ref{lem4}. In fact, given Lemma \ref{lem0}(a), the probability in question is at least
$$\frac{\binom{n/2000-n^{99/100}}{2}-\r-n/2}{\binom{n}{2}}> 10^{-7},$$
where we have subtracted $\r$ to account for some edges of $E_A$ having already been checked. And we have subtracted the size of $M_0$ too.

Now if there is no perfect matching in $H_\r$ then we will have $\m_i=\m_{i-1}+1$ at most $n/2$ times. But from \eqref{match} we see that the probability of this is bounded by $\Pr\brac{\Bin\brac{\r,10^{-7}}\leq n/2}$. It follows that
\beq{eq2}
\Pr(H_\r\text{ has no perfect matching})\leq O(n^{-0.51})+\Pr\brac{\Bin\brac{\r,10^{-7}}\leq n/2}= O(n^{-0.51}).
\eeq
So  with probability $1-O(n^{-0.51})$, $\Psi_2=H_\r$ has a perfect matching. We choose such a matching uniformly at random.

It follows by symmetry that $M_1$ is uniformly random, conditional only on being disjoint from  $M_0$. This will not be true if we condition on various quantities like $\Psi_0,V_0$ etc., but we only make an unconditional claim (except for $M_0$). We will need the following properties of the 2-factor
$$\Pi_0=M_0\cup M_1.$$
\begin{lemma}\label{lem2}
The following hold with probability $1-O(n^{-0.51})$:
\begin{description}
\item[(a)] $M_0\cup M_1$ has at most $10\log_2n$ components.
\item[(b)] There are at most $n_b=\frac{n\log\log\log n}{\log\log n}$ vertices in total in components of size at most $n_c =\frac{200n}{\log n}$.
\end{description}
\end{lemma}
\proofstart
Let
$$\n(m)=\frac{(2m)!}{2^mm!}\text{ = number of perfect matchings of $K_{2m}$.}$$
We observe that if we choose $M_1$ completely independently of $M_0$, then using inclusion-exclusion we see that the probability that $M_0\cap M_1= \emptyset$ is
\beq{incexc}
\sum_{k=0}^{n/2}(-1)^k\binom{n/2}{k}\frac{\n(n/2-k)}{\n(n/2)}.
\eeq
Now for $k$ constant we see that the summand in \eqref{incexc} is asymptotically equal to $\frac{1}{2^kk!}$. Then by truncating the sum in \eqref{incexc} at a large odd integer and using the Bonferroni inequality we see that the sum in \eqref{incexc} is at least $e^{-1/2}-\d$ for any positive $\d$. We will therefore accept that $\Pr(M_0\cap M_1=\emptyset)\geq 1/3$ and then we can inflate the probabilities in \eqref{xx1}, \eqref{xx2} by 3, at most, to handle the conditioning on $M_0\cap M_1=\emptyset$.

(a) We generate a uniform random matching by choosing any unmatched vertex $v$ and pairing it with a random unmtched vertex $w$. Following the argument in \cite{FLu} we note that if  $C$ is the cycle
of $M_0\cup M_1$ that contains vertex 1 then
\beq{eq12}
\Pr(|C|=2k)< \prod_{i=1}^{k-1}\bfrac{n-2i}{n-2i+1}\frac{1}{n-2k+1}<\frac{1}{n-2k+1}.
\eeq
Indeed, consider $M_0$-edge $\set{1=i_1,i_2}\in C$ containing vertex 1. Let $\set{i_2,i_3}\in C$ be the $M_1$-edge containing $i_2$. Then $\Pr(i_3\neq 1)=\frac{n-2}{n-1}$. Assume $i_3\neq 1$ and let $\set{i_3,i_4\neq 1}\in C$ be the $M_0$ edge containing $i_3$. Let $\set{i_4,i_5}\in C$ be the $M_1$-edge containing $i_4$. Then $\Pr(i_5\neq 1)=\frac{n-4}{n-3}$ and so on.

Having chosen $C$, the remaining cycles come from the union of two (random)
matchings on the complete graph $K_{n-|C|}$. It follows from this, by summing \eqref{eq12} over $k\leq n/4$ that 
$$\Pr(|C|<n/2)\leq \sum_{k=1}^{n/4}\frac{1}{n-2k+1}\leq \frac{n}{4}\times \frac2{n}=\frac12.$$ 
Hence, from \eqref{chern1} with $\e=4/5$,
\beq{xx1}
\Pr(\neg(a))\leq \Pr(Bin(10\log_2n,1/2)\leq \log_2n)\leq 2e^{-10\log_2n/3}=O(n^{-0.51}).
\eeq

(b) It follows from \eqref{eq12} that
$$\Pr(|C|\leq n_c)\leq \frac{201}{\log n}.$$
If we generate cycle sizes as in (a) then up until there are fewer than $n_b/2$
vertices left, $\log\n\sim \log n$ where $\n$ is the number of vertices that
need to be partitioned into cycles. It follows that the probability we generate
more than $k=\frac{\log\log\log n\times \log n}{1000\log\log n}$ cycles of size
at most $n_c$ up to this time is bounded by
\beq{xx2}
O(n^{-0.51})+\Pr\brac{Bin\brac{10\log_2n,\frac{201}{\log n}}\geq
k}\leq O(n^{-0.51})+ \bfrac{3000e}{k}^k=O(n^{-0.51}).
\eeq

Thus with probability $1-O(n^{-0.51})$, we have at most
$$\frac{n_b}{2}+kn_c\leq n_b$$
vertices on cycles of length at most $n_b$.
\proofend
\subsection{Phase 2: Increasing minimum cycle length}\label{ics}
In this section, we will use the edges in
$$E_B=\set{e\in E_{t_4}\setminus E_{t_3}:e\cap V_0=\emptyset}$$
to create a 2-factor that contains $M_0$ and in which each cycle has length at least $n_c$. Note that
$$E_B\cap \Psi_1=\emptyset.$$
Note also that 
\begin{lemma}\label{lem4x}
Given $\Psi_1$ and $E_{t_3}$, $E_B$ is a uniformly random $|E_B|$-subset of
$E_3=\binom{V_1}{2}\setminus (\Psi_1\cup E_{t_3})$, where $V_1=[n]\setminus V_0$.
\end{lemma}
\proofstart
This follows from the fact that if we remove any edge of $E_B$ and replace it with any other edge from $E_3$ then $V_0$ is unaffected.
\proofend

We eliminate the small cycles (of length less than $n_c$) one by one (more or less). Let $C$ be a small cycle. We remove an edge $\set{u_0,v_0}\notin M_0$ of $C$. We then try to join $u_0,v_0$ by a sufficiently long $M_1$ alternating path $P$ that begins and ends with edges not in $M_0$. This is done in such a way that the resulting 2-factor contains $M_0$ but has at least one less small cycle. The search for $P$ is done in a breadth first manner from both ends, creating $n^{2/3+o(1)}$ paths that begin at $v_0$ and another $n^{2/3+o(1)}$ paths that end at $u_0$. We then argue that with sufficient probability, we can find a pair of paths that can be joined by an edge from $E_B$ to create the required alternating path.

We proceed to a detailed description. Let
$$V_\t=\set{v\in [n]\setminus V_0:\deg_{E_B}(v)\leq L_0},$$
where for a set of edges $X$ and a vertex $x$,
$\deg_X(x)$ is the number of edges in $X$ that are
incident with $x$.

\begin{lemma}\label{lem3}
The following hold with probability $1-O(n^{-0.51})$:
\begin{description}
\item[(a)] $|V_\t|\leq n^{2/5}$.
\item[(b)] No vertex has 10 or more $G_{t_1}$ neighbors in $V_\t$.
\item[(c)]
If $C$ is a cycle with $|C|\leq n_c$ then $|C\cap V_\t|\leq |C|/200$ in $G_{t_1}$.
\end{description}
\end{lemma}
\proofstart

(a) Let $p=\frac{|E_B|}{|E_3|}\approx \frac{7\log n}{n}$, assuming that $|V_0|=o(n)$. Suppose we replace $E_B$ by a subset $X\subseteq E_3$ with edges included independently with probability $p$. Fix a set $U\subseteq V_1=V\setminus V_0$ of size $\m$. For $v\in U$, now let $d(v,V_1\setminus U)$ denote the number of edges in $X$ incident with $v$ and $V_1\setminus U$. Then, if $n_1=|V_1|=n-o(n)$, 
\begin{align*}
\Pr(d(v,V_1\setminus U)\le L_0, \forall v\in U)&= \brac{\sum_{i=0}^{L_0} \binom{n_1}{i}p^i(1-p)^{n_1-i}}^\m\\
&=(n^{-7/10+(\log 100)/100+o(1)})^\m<n^{-13\m/20}.
\end{align*}
Hence, applying \eqref{notmon}, we have with $\m=n^{2/5}$,
$$\Pr(|V_\t|\geq \m)\leq O(n^{1/2+o(1)})\binom{n}{\m}n^{-13\m/20}\leq O(n^{1/2+o(1)})\brac{\frac{ne}{\m}\cdot n^{-13/20}}^\m=o(n^{-1}).$$

(b) This time we can condition on $\n=n-|V_0|$ and $\m=|\set{e\in E_{t_4}\setminus E_{t_3} :e\cap V_0\neq \emptyset}|\leq n^{99/100}\times 10\log n$. We write $$\Pr(v\text{ violates (b)})\leq
\sum_{S\in\binom{[n-1]}{10}}\Pr(\cA(v,S))\Pr(\cB(v,S)\mid \cA(v,S))
$$
where
\begin{align*}
\cA(v,S)&=\set{N(v)\supseteq S,
\text{ in }G_{t_1}},\\
\cB(v,S)&=\set{w\text{ has at most $L_0$ $E_B$-neighbors in }[n]\setminus
(S\cup\set{v}),\forall w\in S}.
\end{align*}
Applying \eqref{mon} we see that $\Pr(\cA(v,S))\leq 3p_1^{10}$ and then using \eqref{mon} with
\beq{p}
p=\frac{t_4-t_3-\m}{\binom{\n}{2}}\sim \frac{7\log n}{10n}
\eeq
we see that
$$\Pr(\cB(v,S)\mid \cA(v,S))\leq 3\brac{\sum_{k=0}^{L_0}
\binom{\n-11}{k}p^k(1-p)^{\n-11-k}}^{10}$$
and so
\begin{align*}
\prob(v\text{ violates (b)})
& \leq_b \binom{n}{10}p_1^{10}\brac{\sum_{k=0}^{L_0}
\binom{\n-11}{k}p^k(1-p)^{\n-11-k}}^{10}\\
& \leq  (e^{o(1)}\log n\cdot n^{1/10-7/10+o(1)})^{10}\\
&=o(n^{-5}).
\end{align*}
Now use the Markov inequality.

(c) Let $Z$ denote the number of cycles violating the required property. Using \eqref{mon} and $\n$ as in (b) and $p$ as in \eqref{p}, we have
\begin{align}
\expect(Z) & \leq_b\sum_{k=3}^{n_c} \binom{n }{ k}k!p_1^k
\binom{k}{\rdup{\frac{k}{200}}}\brac{\sum_{\ell=0}^{L_0}
\binom{\n-k}{\ell}p^\ell(1-p)^{\n-\ell}}^{\rdup{k/200}}\label{p1p}\\
& \leq \sum_{k=3}^{n_c} (2n)^k \bfrac{\log n+2\log\log n}{n-1}^k n^{-3\rdup{k/200}/5}\\
& = O(n^{-0.51}).
\end{align}
\proofend
Let $\cE_0$ denote the intersection of the high probability events of Lemmas \ref{lem0} and \ref{lem3}. 
\begin{lemma}\label{lem5}
Let $V_1=[n]\setminus V_0$ and let  $|E_B|=\m=\a n\log n,\a=O(1)$ and $|V_1|=\n\geq n-n^{99/100}$.

(a)  If $A\subseteq\binom{V_1}{2}$ with $|A|=a=o(n^{1/2})$ and $X$ is a subset of $\binom{V_1}{2}$ with $|X|=O(n^{99/100}\log n)$ and $A\cap X=\emptyset$, then
\begin{align}
\Pr(E_B\supseteq A\mid \cE_0, X\subseteq E_B)
&=\frac{\binom{\binom{\n}{2}-a-|X|}{\m-a-|X|}}{\binom{\binom{\n}{2}-|X|}{\m-|X|}}
\label{eq5}\\
&=(1+o(1))\bfrac{2\a\log n}{n}^a.\label{eq6}
\end{align}

(b) $A\subseteq\binom{V_1}{2}$ with $|A|=a=o(n^2)$ then
\begin{align}
\Pr(E_B\cap A=\emptyset\mid \cE_0)
&=\frac{\binom{\binom{\n}{2}-a}{\m}}{\binom{\binom{\n}{2}}{\m}}
\label{eq5x}\\
&\leq \exp\set{-\frac{a\m}{\n^2}}.\label{eq6x}
\end{align}
\end{lemma}
\proofstart
(a) Equation \eqref{eq5} follows from Lemma \ref{lem4x}. For equation \eqref{eq6}, we write
$$\frac{\binom{\binom{\n}{2}-a-|X|}{\m-a-|X|}}{\binom{\binom{\n}{2}-|X|}{\m-|X|}}= \bfrac{\m-|X|}{\binom{\n}{2}-|X|}^a\brac{1+O\bfrac{a^2}{\m-|X|}}=\bfrac{\m}{\binom{\n}{2}}^a\brac{1+O\bfrac{a^2}{\m-|X|}+O\bfrac{a|X|}{\m}}.$$
This follows from the fact that in general, if $s^2=o(N)$ then
$$\frac{\binom{N-s}{M-s}}{\binom{N}{M}}=\bfrac{M}{N}^s\brac{1+O\bfrac{s^2}{M}}.$$
(b) Equation \eqref{eq5x} follows as for \eqref{eq5}, and \eqref{eq6x} follows from
$$\frac{\binom{\binom{\n}{2}-a}{\m}}{\binom{\binom{\n}{2}}{\m}}= \prod_{i=0}^{a-1}\frac{\binom{\n}{2}-\m-i}{\binom{\n}{2}-i}.$$
\proofend
By construction, we can apply this lemma to the graph induced by $E_B$ with
\beq{defalpha}
\a\approx\frac{t_4-t_3}{2n\log n}\approx \frac{7}{20}.
\eeq
Let a cycle $C$ of $\Pi_0$ be \emph{small} if its length $|C| < n_c$ and \emph{large} otherwise. Define a near 2-factor to be a graph that is obtained from a 2-factor by removing one edge. A near 2-factor $\G$ consists of a path $P(\G)$ and a collection of vertex disjoint cycles. A 2-factor or a near 2-factor is {\em proper} if it contains $M_0$. We abbreviate proper near 2-factor to PN2F.

\medskip

We will describe a process of eliminating small cycles. In this process we create intermediate proper 2-factors. Let $\G_0$ be a 2-factor and suppose that it contains a small cycle $C$. To begin the elimination of $C$ we choose an arbitrary edge $\set{u_0,v_0}$ in $C\setminus M_0$, where $u_0,v_0\notin V_\t$. This is always possible, since $M_0\cup M_1$ is the union of disjoint cycles of length at least three and because of Lemma \ref{lem3}(c). We delete it, obtaining a PN2F $\G_1$. Here, $P(\G_1) \in \pee(v_0,u_0)$, the set of  $M_1$-alternating paths in $G$ from $v_0$ to $u_0$. Here an  $M_1$-alternating path must begin and end with an edge of $M_1$. The initial goal will be to create a large set of PN2Fs such that each $\G$ in this set has path $P(\G)$ of length at least $n_c$ and the small cycles of $\G$ are a strict subset of the small cycles of $\G_0$. Then we will show that with probability $1-O(n^{-0.51})$, the endpoints of one of the paths in some such $\G$ can be joined by an edge to create a proper 2-factor with at least one fewer small cycle than $\G_0$.

\medskip

This process can be divided into two stages. In a generic step of Stage 1, we take a PN2F $\G$ as above with $P(\G) \in \pee(u_0, v)$ and construct a new PN2F with the same starting point $u_0$ for its path. We do this by considering edges from $E_B$ incident to $v$. Suppose $\set{v,w} \in E_B$ and that the non-$M_0$ edge in $\G$ containing vertex $w$ is $\set{w,x}$. Then $\G' = \G \cup \set{v,w} \setminus \set{w,x}$ is a PN2F with $P(\G') \in \pee(u_0, x)$. We say that $\set{v,w}$ is \emph{acceptable} if 
\begin{enumerate}[(i)]
\item $x,w\notin W$ ($W$ defined immediately below). 
\item $P(\G')$ has length at least $n_c$ and any new cycle created (in $\G'$ but not $\G$) has at least $n_c$ edges.
\end{enumerate}
\medskip

There is an unlikely technicality to be faced. If $\G$ has no non-$M_0$ edge $(x,w)$, then $w= u_0$ and this is accepted if $P(\G')$ has at least $n_c$ edges and it ends the round. When  $P(\G')$ has fewer edges we lose one out of $L_0=\Omega(\log n)$ possible branching choices and this is inconsequential. It is also unlikely, having probability $O(|E_B|/\binom{n}{2})=O(\log n/n)$. We refer to this as event $\cC$ and we remark on it in the proof of Lemma \ref{lem6} below.
\medskip

In addition we define a set $W$ of \emph{used} vertices, where
$$W= V_\t\text{ at the beginning of Phase 2},$$ 
and whenever we look at edges $\set{v,w},\set{w,x}$ (that is, consider using that edge to create a new $\G'$), we add $v,w,x$ to $W$. Additionally, we maintain $|W|=O(n^{99/100})$, or fail if we cannot. Note also that $W$ accumulates as we remove short cycles.

\medskip

We will build a tree $T$ of PN2Fs, breadth-first, where each non-leaf vertex $\G$ yields PN2F children $\G'$ as above. When we stop building $T$ we will have $\nu_L=n^{2/3+o(1)}$ leaves, see \eqref{eq8}. This will end Stage 1 for the current cycle $C$ being removed.

\medskip

We will restrict the set of PN2F's which could be children of $\G$ in $T$ as follows: We restrict our attention to $w\notin W$ with $\set{v,w}\in E_B$ and $\set{v,w}$ acceptable as defined above. Also, we only construct children from the first $\ell_0=L_0/2$ acceptable $\set{v,w}$'s at a vertex $v$. Furthermore we only build the tree down to $\ell_1=\frac{2\log n}{3\log\log n}$ levels. We denote the nodes in the $i$th level of the tree by $S_i$. Thus $S_0=\set{\G_1}$ and $S_{i+1}$ consists of the PN2F's that are obtained from $S_i$ using acceptable edges. In this way we define a tree of PN2F's with root $\G_1$ that has branching factor at most $\ell_0$. Thus,
\beq{eq8}
|S_{\ell_1}|\leq \n_L=\ell_0^{\ell_1}.
\eeq
Now augment $\cE_0$ with the properties claimed in Lemma \ref{lem2}. Then, 
\begin{lemma}\label{lem6}
Conditional on the event $\cE_0$,
$$|S_{\ell_1}| =\n_L$$
with probability $1-o(n^{-3})$.
\end{lemma}
\proofstart
If $P(\G)$ has endpoints $u_0,v$ and $e=\set{v,w}\in E_B$ and $e$ is unacceptable then
(i) $w$ lies on $P(\G)$ and is within distance $n_c$ of an endpoint or (ii) $x\in W$ or $w\in W$ or (iii) $w$ lies on a small cycle or (iv) $w\in V_\t$. Ab initio, there are at least $L_0$ choices for $w$ and we must bound the number of unacceptable choices.

The probability that at least $L_0/10$ vertices are unacceptable due to (iii) is by Lemmas \ref{lem2} and \ref{lem5}(a) at most
\begin{multline}\label{eq7}
(1+o(1))\binom{n_b}{L_0/10}\bfrac{7\log n}{(10+o(1))n}^{L_0/10}\leq
\bfrac{9en_b\log n}{L_0n}^{L_0/10}\\
\leq \bfrac{900e\log\log\log n}{\log\log n}^{L_0/10}=O(n^{-K})
\end{multline}
for any constant $K>0$. In our application of Lemma \ref{lem5}, $X$ is the set of $E_B$-edges incident with $W$ and $A$ is a possible set of $E_B$-edges incident with $v$.

A similar argument deals with conditions (i) and (ii). Lemma \ref{lem3}(b) means that (iv) only requires us to subtract 10.
\medskip

Thus, with (conditional) probability $1-o(n^{-4})$, 
\begin{multline}\label{logn/100}
\text{each vertex of $T$ is incident with at least $\frac{\log n}{100}-\frac{3\log n}{1000}-10-1$ acceptable edges}\\
\text{ and so $|S_{t+1}|\geq \frac{\log n}{200} |S_t|$,}
\end{multline}
for all $t$. (The -1 accounts for the possible occurrence of the event $\cC$). So with (conditional) probability $1-o(n^{-3})$ we have
\beq{eq9}
|S_{\ell_1}| = \n_L
\eeq
as desired. (This assumes that $|W|$ remains $O(n^{99/100})$, see Remark \ref{rem00} below.)
\proofend

Having built $T$, if we have not already made a cycle, we have a tree of PN2Fs and the last level, $\ell_1$ has leaves $\G_i, \ i=1,...,\n_L$, each with a path $P(\G_i)$ of length at least $n_c$. (Recall the definition of an acceptable edge.) Now, perform a second stage which will be like executing $\n_L$-many \emph{Stage 1}'s {\em in parallel} by constructing trees $T_i, \ i=1,...,\n_L $ each of depth $\ell_1$, where the root of $T_i$ is $\G_i$. Suppose for each $i$, $P(\G_i) \in \pee(u_0,v_i)$; we fix the vertex $v_i$ and build paths by first looking at neighbors of $u_0$, for all $i$ (so in tree $T_i$, every $\G$ will have path $P(\G) \in \pee(u,v_i)$ for some $u$).

\medskip

Construct these $\n_L$ trees in Stage 2 by only enforcing the conditions that $x,w \notin W$. This change will allow the PN2Fs to have small paths and cycles. We will not impose a bound on the branching factor either. As a result of this and the fact that each tree $T_i$ begins by considering edges from $E_B$ incident to $u_0$, the sets of endpoints of paths (that are not the $v_is$) of PN2Fs at the same level are the same in each of the trees $T_i,i=1,2,\ldots,\n_L$. That is, for every pair $1\le i<j\le \nu_L$, if $\G_i$ is a node at level $\ell$ of tree $T_i$ and  $P(\G_i) \in \pee(w, v_i)$ for some $w\notin V_\t$ then there exists a node $\G_j$ at level $\ell$ of tree $T_j$, such that $P(\G_j) \in \pee(w, v_j)$. This can be proved by induction, see \cite{CF}. Indeed, let $L_{i,\ell }$ denote the set of end vertices, other than $v_i$, of the paths associated with the nodes at depth $\ell$ of the tree $T_i$, $i=1,2\ldots ,\n_L, \ell = 0,1,\ldots,\ell_1$. Thus $L_{i,0}=\{ u_0\}$ for all $i$. We can see inductively that $L_{i,\ell}=L_{j,\ell}$ for all $i,j,\ell$. In fact if $v\in L_{i,\ell}=L_{j,\ell}$ then $\set{v,w}\in E_B$  is acceptable for some $i$ means that $w\notin W$ (at the start of the construction of level $\ell+1$) and hence if $\set{w,x}$ is the non-$M_0$ edge for this $i$ then $x\notin W$ and it is the non-$M_0$ edge for all $j$. In which case $\set{v,w}$ is acceptable for all $i$ and we have $L_{i,\ell +1}=L_{1,\ell +1}$.

\medskip
The set of trees $T_i, i=1,...,\nu_L$, will be succesfully constructed (i.e. have exactly $\n_L$ leaves) with probability $1-o(1/n^3)$ and with a similar probability the number of nodes in each tree is at most $(100\log n)^{\ell_1}=n^{2/3+o(1)}$. Here we use the fact that the maximum degree in $G_{t_1}\leq 100\log n$ with this probability, see \eqref{10logn}. However, some of the trees may use unacceptable edges, and so we will ``prune'' the trees by disallowing any node $\G$ that was constructed in violation of any of those conditions. Call tree $T_i$ GOOD if it still has at least $\n_L$ leaves remaining after pruning and BAD otherwise. Notice that
$$\Pr(\exists\ i:T_i \text{ is BAD}\mid\cE_0) = o\bfrac{\n_L}{n^3}=o(n^{-2}).$$
Here the $o(1/n^3)$ factor is the one promised in Lemma \ref{lem6}.

Finally, consider the probability that there is no $E_B$ edge from any of the $n^{2/3+o(1)}$ endpoints found in Stage 1 to any of the $n^{2/3+o(1)}$ endpoints found in Stage 2. At this point we will have only exposed the $E_B$-edges of $\Pi_0$ incident with these endpoints. So if for some $k\leq \n_L$ we examine the (at least) $\log n/100$ edges incident to $v_1,v_2,\ldots,v_k$, then from Lemma \ref{lem5}(b), with $X$ equal to the $E_B$-edges incident with $W$ and $A$ equal to the set of pairs $(v_i,w),i\leq k$ where $w$ is a leaf of some $T_i,1\leq i\leq \n_L$, we see that the probability we fail to close a cycle and produce a proper 2-factor is at most
$$\exp\set{-\frac{k\times n^{2/3+o(1)}\time \a n\log n}{\binom{\n}{2}}}.$$
Thus taking $k=n^{1/3+o(1)}$ suffices to make the failure probability $o(n^{-2})$. (If we have $n^{\g}$ endpoints here, then we need $k$ to be $\omega(n^{1-\g})$.)  Also, this final part of the construction only contributes $n^{1/3+o(1)}$ to $W$, viz. $v_1,v_2,\ldots,v_k$ and $O(k\log n)$ of their neighbors. Our choice of $k=n^{1/3+o(1)}$ and $n^{2/3+o(1)}$ for tree size makes this probability small and controls the size of $W$. There are other choices, this is just one of them.

Therefore, the probability that we fail to eliminate a particular small cycle $C$ is $o(n^{-2})$ and then given $\cE_0$, the probability that Phase 2 fails is $o(\log n/n^2)=o(1)$.
\begin{remark}\label{rem00}
We should check now that w.h.p. $|W|=O(n^{99/100})$ throughout Phase 2. It starts out with at most $n^{99/100}+n^{2/5}$ vertices (see Lemmas \ref{lem0}(a) and \ref{lem3}(a)) and we add $O(n^{2/3+o(1)}\times \log n)$ vertices altogether in this phase.
\end{remark}
So we conclude:
\begin{lemma}\label{lem7}
The probability that Phase 2 fails to produce a proper 2-factor with
minimum cycle length at least $n_c$ is $O(n^{-0.51})$.
\end{lemma}
\proofend


\subsection{Phase 3: Creating a Hamilton cycle}\label{CHC}

By the end of Phase 2, we will with probability $1-O(n^{-0.51})$ have found a proper 2-factor
with all cycles of length at least $n_c$. Call this subgraph $\Pi^*$.

\medskip

In this section, we will use the edges in
$$E_C=\set{e\in E_{t_0}\setminus (E_{t_4}\cup E(\Psi_1)):e\cap V_0=\emptyset}$$
to turn $\Pi^*$ into a Hamilton cycle that contains $M_0$, w.h.p. It is basically a second moment calculation with a twist to keep the variance under control. We note that Lemma \ref{lem5} continues to hold if we replace $E_B$ by $E_C$ and $\a$ by $\frac{1}{20}+o(1)$.

\medskip

Arbitrarily assign an orientation to each cycle. Let $C_1,...,C_k$ be the cycles of $\Pi^*$ (note that if $k=1$ we are done) and let $c_i = \rdup{|C_i\setminus W|/2}$. Then $c_i \geq \frac{n_c}{2}-O(n^{99/100}) \geq \frac{99 n}{\log n}$ for all $i$. Let $a= \frac{n}{\log n}$ and $m_i = 2\lfloor \frac{c_i}{a} \rfloor+1$ for all $i$ and $m = \sum_{i=1}^k m_i$. We arbitrarily orient the cycles $C_1,\ldots,C_k$. Then from each $C_i$, we will consider choosing $m_i$ edges $\set{v,w}$ such that $v,w \in C_i\setminus W$ and $v$ is the head of a non-$M_0$ arcs after the arbitrary orientation of the cycles. We then delete these $m$ arcs and replace them with $m$ others to create a proper Hamilton cycle. We use a second moment calculation to show that such a substitution is possible w..h.p.

\medskip

Given such a deletion of edges, re-label the broken arcs as $(v_j ,u_j), j \in [m]$ as follows: in cycle $C_i$ identify the lowest numbered vertex $x_i\in [n]$ which loses a cycle edge directed out of it. Put $v_1=x_1$ and then go round $C_1$ defining $v_2,v_3,\ldots v_{m_1} $ in order. Then let $v_{m_1+1}=x_2$ and so on.
We thus have $m$ path sections $P_j\in {\cal P}(u_{\f(j)},v_j) $ in $\Pi^{*}$ for some permutation $\f$.

It is our intention to rejoin these path sections of $\Pi^{*}$ to make a Hamilton cycle using $E_C$, if we can. Suppose we can. This defines a permutation $\r$ on $[m]$ where $\r (i) = j$ if $P_i$ is joined to $P_j$ by $(v_i,u_{\phi (j)})$, where $\r\in H_m$, the set of cyclic permutations on $[m]$. We will use the second moment method to show that a suitable $\r$ exists w.h.p. A technical problem forces a restriction on our choices for $\r$. This will produce a variance reduction in a second moment calculation, as explained in \eqref{s=t}.

Given $\r$ define $\lambda=\f\r$. In our analysis we will restrict our attention to $\r\in R_{\f} =\{ \r \in H_m : \f \r \in H_m \}$. If $\r\in R_{\f}$ then we have not only constructed a Hamilton cycle in $\Pi^{*}\cup E_C$, but also in the \emph{auxiliary digraph} $\LL$, whose edges are $(i, \lambda(i))$.

\medskip

The following lemma is from \cite{CF1}. The content is in the lower bound. It shows that there are still many choices for $\r$ and it is needed to show that the expected number of possible re-arrangements of path sections grows with $n$.
\begin{lemma}\label{lem8}
$(m-2)! \leq |R_{\phi}| \leq (m-1)!$
\end{lemma}
\medskip

Let $H$ be the graph induced by the union of $\Pi^*$ and $E_C$. In the following lemma we drop the requirement that events occur with probability $1-O(n^{-0.51})$. This requirement was used to handle issues related to $M_0$ and the edges chosen. At this point these issues no longer matter and w.h.p. takes its usual meaning.
\begin{lemma}
$H$ contains a Hamilton cycle w.h.p.
\end{lemma}

\begin{proof}
Let $X$ be the number of Hamilton cycles in $G$ that can be obtained by removing the edges described above and
rearranging the path segments generated by $\phi$ according to those in $\rho \in R_{\phi}$ and connecting the path segments using edges in $H$.

\medskip

We will use the inequality $\prob(X>0) \geq \frac{\eee(X)^2}{\eee(X^2)}$ to show that such a Hamilton cycle exists with the required probability.

\medskip

The definition of $m_i$ gives us $\frac{n-|W|}{a}-k \leq m \leq \frac{n-|W|}{a}+k$
and so $1.99\log n \leq m \leq 2.01\log n$.
Additionally we will use $k \leq \frac{n}{n_c}=\frac{\log n}{200}$, $m_i \geq 199$ and
$\frac{c_i}{m_i} \geq \frac{a}{2.01}$ for all $i$.

\medskip

{From} Lemmas \ref{lem5} and \ref{lem8}, we have, with $\a=1/20+o(1)$,
\begin{align}
\eee(X) & \geq (1-o(1))\left( \frac{2\a\log n}{n}\right)^m (m-2)!
\prod_{i=1}^k \binom{c_i}{m_i} \label{eq10a}\\
& \geq  \frac{1-o(1)}{m^{3/2}} \bfrac{2m\a\log n}{en}^m
\prod_{i=1}^k \brac{ \bfrac{c_ie^{1-1/10m_i}}{m_i^{1+(1/2m_i)}}^{m_i}
\bfrac{1-2m_i^2/c_i}{\sqrt{2\pi}}}\label{eq10aa}\\
&=\frac{(1-o(1))e^{-k/10}(2\p)^{-k/2}}{m^{3/2}} \bfrac{2m\a\log n}{en}^m
\prod_{i=1}^k \bfrac{c_ie}{m_i^{1+(1/2m_i)}}^{m_i}
\end{align}
where to go from \eqref{eq10a} to \eqref{eq10aa} we have used the approximation $(m-2)!\geq m^{-3/2}(m/e)^m$ and
$$\binom{c_i}{m_i}\geq \frac{c_i^{m_i}(1-2m_i^2/c_i)}{m_i!}\text{ and }
m_i!\leq \sqrt{2\p m_i}\bfrac{m_i}{e}^{m_i}e^{1/10m_i}.$$

{\bf Explanation of \eqref{eq10a}:} We choose the arcs to delete in
$\prod_{i=1}^k \binom{c_i}{m_i}$
ways and put them together as explained prior to Lemma \ref{lem8}
in at least $(m-2)!$
ways. The probability that the required edges exist in $E_C$ is
$(1+o(1))\bfrac{2\a\log n}{n}^m$,
from Lemma \ref{lem5}.

\medskip
Continuing, we have
\begin{align}
\eee(X)& \geq  \frac{(1-o(1))(2\pi)^{-k/2}e^{-k/10}}{m^{3/2}}
\left( \frac{2m\a\log n}{en} \right)^m
\prod_{i=1}^k \left( \frac{c_ie}{(1.02)m_i} \right)^{m_i} \nonumber \\
& \geq  \frac{(1-o(1))(2\pi)^{-k/2}}{n^{1/2000}m^{3/2}}
\left( \frac{2m\a\log n}{en} \right)^m \bfrac{ea}{2.01\times 1.02}^m \nonumber \\
& \geq  \frac{1-o(1)}{n^{1/1000}m^{3/2}}
\left( \frac{\log n}{30} \right)^m \nonumber\\
& \to\infty.
\end{align}

\medskip

Let $M ,M^{\prime}$ be two sets of selected edges which have been deleted
in $\Pi^{*}$ and whose path sections have been
re-arranged into Hamilton cycles according to $\r , \r^{\prime}$ respectively.
Let $N,N'$ be the corresponding sets of edges which have been added to make
the Hamilton cycles. Let $\Omega$ denote the set of choices for $M$ (and $M'$.)

Let $s=|M\cap M'|$ and $t=|N\cap N'|$. Now $t\leq s$ since if $(v,u)\in
N\cap N'$ then there must be a unique $(\tilde{v},u)\in M\cap M'$ which
is the unique $\Pi^{*}$-edge into $u$. It is shown in \cite{CF1} that 
\beq{s=t}
\text{$t=s$ implies $t=s=m$ and $(M,\rho )=(M',\rho ')$.}
\eeq
(This removes a large term from the second moment calculation). Indeed, suppose then that $t=s$ and $(v_i,u_i)\in M\cap M'$.
Now the edge $(v_i,u_{\lambda (i)})\in N$ and since $t=s$ this edge must
also be in $N'$. But this implies that $(v_{\lambda (i)},u_{\lambda (i)})\in M'$
and hence in $M\cap M'$. Repeating the argument we see that
$(v_{\lambda ^k(i)},u_{\lambda ^k(i)})\in M\cap M'$ for all
$k\geq 0$. But $\lambda$ is cyclic and so our claim follows.

If $\langle s,t \rangle$ denotes the case where $s= |M\cap M'|$ and
$t=|N\cap N'|$, then
\begin{align*}
\eee(X^2) & \leq \eee(X) + (1+o(1))\sum_{M\in \Omega} \bfrac{2\a\log n}{n}^m
\sum_{\substack{M'\in \Omega\\ N'\cap N = \emptyset}} \left(\frac{2\a\log n}{n}\right)^m \\
& + (1+o(1)) \sum_{M\in \Omega} \left(\frac{2\a\log n}{n}\right)^m \sum_{s=2}^m
\sum_{t=1}^{s-1}
\sum_{\substack{M'\in\Omega\\ \langle s,t \rangle}} \left(\frac{2\a\log n}{n}\right)^{m-t} \\
& = \eee(X) + E_1 + E_2 \ \text{say.}
\end{align*}
Note that $E_1 \leq (1+o(1))\eee(X)^2$.

\medskip

Now, with $\s_i$ denoting the number of common $M\cap M'$ edges selected from $C_i$,
\[
E_2 \leq E(X)^2 \sum_{s=2}^m \sum_{t=1}^{s-1} \binom{s}{t}
\bigg[ \sum_{\sigma_1+...+\sigma_k=s} \
\prod_{i=1}^k \frac{\binom{m_i}{\sigma_i}
\binom{c_i-m_i}{m_i - \sigma_i}}{\binom{c_i}{m_i}} \bigg] \frac{(m-t-1)!}{(m-2)!}
\left(\frac{n}{2\a\log n}\right)^t.
\]
{\bf Some explanation:} There are $\binom{s}{t}$ choices for $N\cap N'$, given $s$ and $t$. Given $\s_i$ there are $\binom{m_i}{\sigma_i}$ ways to choose $M\cap M'$ and $\binom{c_i-m_i}{m_i - \sigma_i}$ ways to choose the rest of $M'\cap C_i$. After deleting $M'$ and adding $N\cap N'$ there are at most $(m-t-1)!$ ways of putting the segments together to make a Hamilton cycle.

We see that
\[
\frac{\binom{c_i-m_i}{m_i - \sigma_i}}{\binom{c_i}{m_i}} \leq \frac{\binom{c_i}{m_i - \sigma_i}}{\binom{c_i}{m_i}}=\frac{m_i(m_i-1)\cdots(m_i-\s_i+1)}{(c_i-m_i+1)\cdots (c_i-m_i+\s_i)}
\leq (1+o(1))\left(\frac{2.01}{a}\right)^{\sigma_i}
\text{exp}\left\{-\frac{\sigma_i(\sigma_i-1)}{2m_i}\right\}.
\]

Also, Jensen's inequality, applied twice implies that 
$$\sum_{i=1}^k \frac{\sigma_i^2}{2m_i}=\brac{\sum_{i=1}^k\s_i^2}\cdot\brac{\sum_{i=1}^k\frac{\s_i^2}{\sum_{i=1}^k\s_i^2} \frac{1}{2m_i}} \geq \frac{s^2}{k}\cdot \frac{k}{2m}=\frac{s^2}{2m} \ \text{for } \sigma_1+...+\sigma_k = s.$$
Furthermore,
$$\sum_{i=1}^k \frac{\sigma_i}{2m_i} \leq \frac{k}{2} \  \text {and} \\
\sum_{\sigma_1+...+\sigma_k=s} \ \prod_{i=1}^k \binom{m_i}{\sigma_i} = \binom{m}{s}.$$
Using these approximations, we have
$$\sum_{\sigma_1+...+\sigma_k=s} \
\prod_{i=1}^k \frac{\binom{m_i}{\sigma_i}
\binom{c_i-m_i}{m_i - \sigma_i}}{\binom{c_i}{m_i}}\leq
e^{(1+o(1))k/2}\exp\left\{-\frac{s^2}{2m}\right\} \left( \frac{2.01}{a} \right) ^s
\binom{m}{s}.$$

So we can write
$$\frac{E_2}{\eee(X)^2}  \leq
e^{(1+o(1))k/2}\sum_{s=2}^m \sum_{t=1}^{s-1}\binom{s}{t}
\exp\left\{-\frac{s^2}{2m}\right\}\left( \frac{2.01}{a} \right) ^s
\binom{m}{s}\frac{(m-t-1)!}{(m-2)!}
\left( \frac{n}{2\a\log n}\right) ^t.$$
We approximate
$$\binom{m}{s}\frac{(m-t-1)!}{(m-2)!}\leq
C_1\frac{m^s}{s!}\bfrac{m-t-1}{e}^{m-t-1}\bfrac{e}{m-2}^{m-2}\leq C_2\frac{m^s}{s!}\frac{e^{t}}{m^{t-1}},$$
for some constants $C_1,C_2>0$.

Substituting this in, we obtain,
\begin{align*}
\frac{E_2}{\eee(X)^2}
&\leq_b n^{1/399} m\sum_{s=2}^m \left( \frac{2.01}{a}\right)^s \frac{m^s}{s!}\text{exp}\left\{ - \frac{s^2}{2m}\right\}  \sum_{t=1}^{s-1} \binom{s}{t}\left(\frac{en}{2\a m\log n}\right)^t \\
&\leq n^{1/399} m\sum_{s=2}^m \left( \frac{2.01}{a}\right)^s \frac{m^s}{s!}\text{exp}\left\{ - \frac{s^2}{2m}\right\} \times 2m\bfrac{en}{2\a m\log n}^{s-1}\\
& \leq_b \frac{m^2}{n^{.99}} \sum_{s=2}^\infty \left( \frac{(2.01)en\text{ exp}\{-s/2m\} }{2\a a\log n}\right)^s \frac{1}{s!} \\
&\leq \frac{m^2}{n^{.99}} \sum_{s=2}^\infty\frac{30^s}{s!}\\
& =O(n^{-9/10}).
\end{align*}
Combining things, we get
$$\eee(X^2)  \leq \eee(X) + \eee(X)^2(1+o(1)) + \eee(X)^2 n^{-9/10}$$
and so 
$$\frac{(\eee X)^2}{\eee(X^2)} \geq \frac{1}{\frac{1}{\eee X} +1+o(1) + n^{-9/10}} \longrightarrow 1$$
as $n \rightarrow \infty$, as desired.
\end{proof}

\begin{remark}\label{disjoint}
We now consider the case where we are given $M_0$ and we must choose edges disjoint from $M_0$. 
\begin{enumerate}[(a)]
\item If we choose $t_1$ edges independently of $M_0$ then the probability they are disjoint from $M_0$ is, where $N=\binom{n}{2}$,
\beq{disj}
\frac{\binom{N-n/2}{t_1}}{\binom{N}{t_1}}=\prod_{i=0}^{t_1-1}\brac{1-\frac{n}{2(N-i)}}\geq \exp\set{-\sum_{i=0}^{t_1-1}\frac{n}{2(N-i)}+O\bfrac{t_1n^2}{N^2}}=n^{-1/2+o(1)}.
\eeq
\item We have shown that if we generate $t_1$ edges independent of $M_0$ then conditional on $t_0\leq\t_1\leq t_1$ we have that with probability $1-O(n^{-0.51})$ there is a perfect matching in $E_{\t_1}\setminus M_0$.
\item If we only choose from edges not in $M_0$ then the distribution of the edges we choose is the same as simply conditioning on $E_{t_1}\cap M_0=\emptyset$.
\end{enumerate}
It follows from (a),(b),(c) that if we avoid $M_0$ then we will still w.h.p. find a perfect matching $M_1$. Indeed, letting $\cA=\set{M_1\text{ exists}},\cB=\set{E_{t_1}\cap M_0=\emptyset}$ andn $\cT=\set{t_0\leq\t_1\leq t_1}$ as before, we have
\beq{f*k}
\Pr(\bar{\cA}\mid\cB)=\frac{\Pr(\bar{\cA}\cB\cT)}{\Pr(\cB)}+ \frac{\Pr(\bar{\cA}\cB\bar{\cT})}{\Pr(\cB)}\leq \frac{\Pr(\bar{\cA}\mid\cT)}{\Pr(\cB)\Pr(\cT)}+ \Pr(\bar{\cT}\mid\cB).
\eeq
Now
$$\frac{\Pr(\bar{\cA}\mid\cT)}{\Pr(\cB)\Pr(\cT)} =\frac{O(n^{-0.51})}{\Omega(n^{-0.5+o(1)})(1-o(1))}=o(1)$$
and this deals with the first term on the RHS of \eqref{f*k}.

For the second term  on the RHS of \eqref{f*k} we have
$$\Pr(\bar{\cT}\mid\cB)\leq n\frac{\binom{\binom{n}{2}-\frac12n-(n-2)}{t_1}}{\binom{\binom{n}{2}-\frac12n}{t_1}}\leq n\brac{1-\frac{n-2}{\binom{n}{2}-\frac12n}}^{t_1}=o(1).$$
It follows that $\Pr(\bar{\cA}\mid\cB)=o(1)$. The remainder of the proof that there is a Hamilton cycle containing $M_0$ goes through with minor changes that reflect the fact that we do not choose edges of $M_0$.
\end{remark}

\subsection{Proof of Corollary \ref{cor1}}
We begin the proof by replacing the sequence $E_0,E_1,\ldots,E_m,\ldots$ by $E_0',E_1',\ldots,E_m',\ldots,$ where the edges of $E_m'=\set{e_1',e_2',\ldots,e_m'}$ are randomly chosen \emph{with replacement}. This means in particular that $e_m$ is allowed to be a member of $E_{m-1}'$. We let $G_m'$ be the graph $([n],E_m')$.

If an edge appears a second time, it will keep its original color. We let $R$ denote the set of edges that get repeated, up to time $\t_{1,1}$. Note that 
\beq{tau11}
2t_0\leq \t_{1,1}\leq 2t_1\ w.h.p.
\eeq
since the Chernoff bounds imply that w.h.p. we there at most $t_0+O(n^{1/2}\log n)$ edges of each color at time $t_0$ and at least $t_1-O(n^{1/2}\log n)$ edges of each color at time $t_1$. Note that if $e_{\t_{1,1}}=\set{v,w}\in R$ then $v$ or $w$ is isolated in $G_{\t_{1,1}-1}^{(b)}$ or $G_{\t_{1,1}-1}^{(w)}$.
\beq{(b)}
\Pr(e_{\t_{1,1}}\in R)\leq 4\Pr(\exists e=\set{v,w}\in R:v\text{ has black degree 1 at time $\t_{1,1}$})=o(1).
\eeq
{\bf Explanation:}
The factor 4 comes from $v$ or $w$ having black or white degree one at time $\t_{1,1}$. Next
suppose first that $e_{\t_{1,1}}=\set{v,w}$ and that $v$ has black degree zero in
$G_{\t_{1,1}-1}$ and $w$ also has black degree zero in $G_{\t_{1,1}-1}$. Now w.h.p. there is no white edge joining $v$ and $w$ and so $e_{\t_{1,1}}\notin R$. Indeed, the probability of this event can be bounded by
$$o(1)+\sum_{t=2t_0}^{2t_1}\binom{n}{2}\frac{1}{\binom{n}{2}} \brac{\brac{1-\frac{n-1}{2\binom{n}{2}}}^{t} \brac{1-\frac{n-1}{2\binom{n}{2}}}^{t-1}}\leq o(1)+2t_1\bfrac{\log^2n}{n}^2=o(1).$$
The $o(1)$ accounts for $\t_{1,1}$ not being in the interval $[2t_0,2t_1]$. The factor $\binom{n}{2}$ accounts for the choice of ${v,w}$. The factor $1/\binom{n}{2}$ is the probability that the $t$th edge is $\set{v,w}$ and the final product accounts the black degree of both $u,v$ being zero.

Now suppose that $e_{\t_{1,1}}=\set{v,w}$ and that $v$ has black degree zero in $G_{\t_{1,1}-1}$ and $w$ has positive black degree in $G_{\t_{1,1}-1}$. An argument similar to that given for Lemma \ref{lem0}(g) shows that w.h.p. the maximum white degree in $G_{2t_1}'$ is $O(\log n)$. There are $n-1$ choices for $w$,
of which $O(\log n)$ put $e_{\t_{1,1}}$ into $R$. So $e_{\t_{1,1}}$ has an
$O(\log n/n)$ chance of being in $R$. This verifies \eqref{(b)}.

\medskip
At time $m=\t_{1,1}$ the graphs $G^{(b)'}_m,G^{(w)'}_m$ will w.h.p. contain perfect matchings, see \cite{ER-M}. That paper does not allow repeated edges, but removing them enables one to use the result claimed. Here we use the fact that w.h.p. there are only $O(\log^2n)$ repeated edges, (as explained below), they are far apart, and are not incident to any low degree vertices. Thus any argument based on expansion goes through without difficulty. We choose perfect matchings $M_B,M_W$ uniformly at random from $G^{(b)'}_{\t_{1,1}}$, $G^{(w)'}_{\t_{1,1}}$ respectively. Thus by symmetry, each is a random perfect matching disjoint from its oppositely colored perfect matching.

\medskip
We couple the sequence $G_1,G_2,\ldots,$ with the sequence
$G_1',G_2',,\ldots,$ by ignoring repeated edges in the
latter. Thus $G_1',G_2',\ldots,G_m'$ is coupled with a sequence
$G_1,G_2,\ldots,G_{m'}$ where $m'\leq m$. 
It follows
from \eqref{(b)} that w.h.p. the coupled processes stop with the same edge.
Furthermore, they stop with two matchings $M_B,M_W$, independently chosen. We can then begin analysing Phase 2 and Phase 3 within this context.

\medskip
We will prove that
\beq{(a)}
\Pr(M_B\cap R=\emptyset)\geq n^{-1/2-o(1)}.
\eeq
Corollary \ref{cor1} follows from this. If $M_B\cap R=\emptyset$ then the white edges are chosen conditional on being disjoint from $M_B$. It follows from \eqref{(a)} and the fact that Phases 1 and 2 succeed with probability $1-O(n^{-0.51})$ (i.e. when ignoring the conditioning, $M_B\cap R=\emptyset$) that they succeed w.h.p. conditional on $M_B\cap R=\emptyset$. 

Phase 3 succeeds w.h.p. even if we avoid using edges in $R$. We have already carried out calculations with an arbitrary set of $O(n^{99/100}\log n)$ edges that must be avoided. The size of $R$ is dominated by a binomial $Bin(O(n\log n),O(n^{-1}\log n))$ and so $|R|=O(\log^2n)$ w.h.p. So avoiding $R$ does not change any calculation in any significant way. In other words, we can w.h.p. find a zebraic Hamilton cycle in $G_m'$.

\medskip
Finally note that the Hamilton cycle we obtain is zebraic.

{\bf Proof of \eqref{(a)}:} $R$ is a uniformly random set, given its size and it is independent of $M_B$. Indeed, we can repeat edges arbitrarily without changing $M_B$.
Let $t_B$ be the number of black edges, then
$$\Pr(M_B\cap R=\emptyset\mid t_B)\geq\brac{1-\frac{n/2}{N}}^{t_B}\geq \exp\set{-t_B\brac{\frac{1}{n}+O\bfrac{1}{n^2}}}.$$
{\bf Explanation of first inequality:} Each choice of black edge has at most an $\frac{n/2}{N}$ chance of repeating an edge of $M_B$, regardless of previously seen edges.

To remove the conditioning, we take expectations and then by convexity
$$\E\brac{\exp\set{-t_B\brac{\frac{1}{n}+O\bfrac{1}{n^2}}}}\geq\exp\set{- \E(t_B)\brac{\frac{1}{n}+O\bfrac{1}{n^2}}}\geq n^{-1/2-o(1)}$$
since $\E(t_B)\sim \frac12n\log n$. This proves \eqref{(a)}.

\section{Proof of Theorem \ref{th4}}\label{th4proof}
For a vertex $v\in [n]$ we let its \emph{black} degree $d_b(v)$ be the number of black
edges incident with $v$ in $G_{t_0}$. We define its \emph{white} degree $d_w(v)$ analogously. Let a vertex be \emph{large} if $d_b(v),d_w(v)\geq L_0$ and \emph{small} otherwise.

We first show how to construct zebraic paths between a pair $x,y$ of large vertices. We can in fact construct paths, even if we decide on the color of the edges incident with $x$ and $y$. We do breadth first searches from each vertex, alternately using black and white edges, constructing search trees $T_x,T_y$. We build trees with $n^{2/3+o(1)}$ leaves and then argue that we can connect the leaves with a correctly colored edge. We then find paths between small vertices and other vertices by piggybacking on the large to large paths.

We will need the following structural properties:
\begin{lemma}\label{lem12}
The following hold w.h.p.:
\begin{description}
\item[(a)] No set $S$ of at most 10 vertices that is connected
in $G_{t_1}$ contains three small vertices.
\item[(b)] Let $a$ be a positive integer, independent of $n$. No set of vertices $S$, with $|S|=s\leq aL_1,L_1=\frac{\log n}{\log\log n}$, contains more than $s+a$ edges in $G_{t_1}$.
\item[(c)] There are at most $n^{2/3}$ small vertices in $G_{t_0}$.
\item[(d)] There are at most $\log^3n$ isolated vertices in $G_{t_0}$.
\end{description}
\end{lemma}
\proofstart
(a)
We say that a vertex is a \emph{low color vertex} if it is incident in $G_{t_1}$ to at most $L_\e=(1+\e)L_0$ edges
of one of the colors, where $\e$ is some sufficiently small positive constant.
Furthermore, it follows from \eqref{mon} that
\begin{align}
&\Pr(\exists\ \text{a connected $S$ in $G_{n,t_1}$ with three low color vertices}) \nonumber\\
&\leq \sum_{k=3}^{10}\binom{n}{k}k^{k-2}\frac{\binom{N-k+1}{t_1-k+1}}{\binom{N}{t_1}} \binom{k}{3}\Pr(\text{vertices 1,2,3 are low color }\mid [k] \text{ is a connected set})\label{po1}\\
&\leq_b \sum_{k=3}^{10}\binom{n}{k}k^{k-2}\frac{\binom{N-k+1}{t_1-k+1}}{\binom{N}{t_1}}\binom{k}{3}
\brac{2\sum_{\ell=0}^{L_\e}\binom{n-k}{\ell}\bfrac{p_1}{2}^\ell\brac{1-\frac{p_1}{2}}^{n-k-\ell}}^3\label{po2}\\
&\leq_b \sum_{k=3}^{10}n^k\bfrac{t_1}{N}^{k-1}(n^{-0.45})^3\nonumber\\
&\leq_b \sum_{k=3}^{10}n^k \bfrac{\log n}{n}^{k-1}(n^{-0.45})^3\nonumber\\
&=o(1).\nonumber
\end{align}
{\bf Explanation of \eqref{po1},\eqref{po2}:} Having chosen our tree, $\frac{\binom{N-k+1}{t_1-k+1}}{\binom{N}{t_1}}$ is the probability that this tree exists in $G_{t_1}$. Condition on this and choose three vertices. The final $(\cdots)^3$ in \eqref{po2} bounds the probability of the event that 1,2,3 are low color vertices in $G_{n,p_1}$. This event is monotone decreasing when restricted to the edges of a fixed color, given the conditioning. So we can use \eqref{mon} to replace $G_{n,t_1}$ by $G_{n,p_1}$ here.

\medskip
Now a simple first moment calculation shows that w.h.p. each vertex in
$[n]$ is incident with less than $\log n/(\log\log n)^{1/2}$ edges of $E_{t_1}\setminus
E_{t_0}$. Indeed, the number of such edges incident with a fixed vertex $v$ is dominated by the binomial $Bin(t_1-t_0,2/n)=Bin(2n\log\log n,2/n)$. And then
$$\Pr(\exists v)\leq n\binom{2n\log\log n}{\log n/(\log\log n)^{1/2}}\bfrac{2}{n}^{\log n/(\log\log n)^{1/2}} \leq n\bfrac{4e(\log\log n)^{3/2}}{\log n}^{\log n/(\log\log n)^{1/2}}=o(1).$$
 Hence, for (a) to fail, there would have to be a
relevant set $S$ with three vertices, each incident in $G_{t_1}$ with at most $(1+o(1))L_0$ edges of
one of the colors, contradicting the above.

(b)
We will prove something slightly stronger. Suppose that $p=\frac{K\log n}{n}$
where $K>0$ is arbitrary. We will show this result for $G_{n,p}$.
The result for this lemma follows from when $K=1+o(1)$ and from \eqref{mon}.
We get
\begin{align*}
\Pr(\exists\ S)&\leq_b \sum_{s\geq 4}^{aL_1}\binom{n}{s}\binom{\binom{s}{2}}{s+a+1}
p^{s+a+1}\\
&\leq_b  \sum_{s\geq 4}^{aL_1}\brac{\frac{ne}{s}\cdot\frac{sep}{2}}^s(sep)^{a+1}\\
&\leq_b (Ke^2\log n)^{aL_1}\bfrac{\log^2n}{n}^{a+1}\\
&\le n^{o(1)}\left(\frac{\log^{3+L_1}n}{n}\right)^a\,\frac{\log^2n}{n} \\
&=o(1).
\end{align*}

(c) Using \eqref{mon} we see that if $Z$ denotes the number of small vertices then
$$\E(Z)\leq_b n\sum_{k=0}^{L_0}\bfrac{p_0}{2}^k \brac{1-\frac{p_0}{2}}^{n-1-k}
\leq n^{0.55}.$$
We now use the Markov inequality.

(d) Using \eqref{mon} we see that the expected number of isolated vertices in
$G_{t_0}$ is $O(\log^2n)$. We now use the Markov inequality.
\proofend
Now fix a pair of large vertices $x<y$. We will define sets
$\S{b}_i(z),\S{w}_i(z),
i=0,1,\ldots,\ell_1$, $z=x,y$.
Assume w.l.o.g. that $\ell_1$ is even.
We
let $\S{b}_0(x)=\S{w}_0(x)=\set{x}$ and then $\S{b}_1(x)$ (resp. $\S{w}_1(x)$)
is the set consisting of the first $\ell_0$ black
(resp. white) neighbors of $x$ in $G_{t_0}$. We will use the notation
$\S{c}_{\leq i}(x)=\bigcup_{j=1}^i\S{c}_j(x)$ for $c=b,w$.
We now iteratively define for $i=0,1,\ldots,(\ell_1-2)/2$.
\begin{align*}
\hS{b}_{2i+1}(x)&=\set{v\notin \S{b}_{\leq 2i}(x): v\neq y
\text{ is joined by a black $G_{t_0}$-edge to a vertex in
}\S{b}_{2i}(x)}.\\
\S{b}_{2i+1}(x)&=\text{the first $\ell_0^i$ members of
$\hS{b}_{2i+1}(x)$}.\\
\hS{b}_{2i+2}(x)&=\set{v\notin \S{b}_{\leq 2i+1}: v\neq y
\text{ is joined by a white $G_{t_0}$-edge to a vertex in
}\S{b}_{2i+1}(x)}.\\
\S{b}_{2i+2}(x)&=\text{the first $\ell_0^i$ members of
$\hS{b}_{2i+2}(x)$}:
\end{align*}
We then define, for $i=0,1,\ldots,(\ell_1-2)/2$.
\begin{align*}
\hS{w}_{2i+1}(x)&=\set{v\notin (\S{b}_{\leq \ell_1}(x)\cup \S{w}_{\leq 2i}(x)): v\neq y
\text{ is joined by a white $G_{t_0}$-edge to a vertex in
}\S{w}_{2i}(x)}\\
\S{w}_{2i+1}(x)&=\text{the first $\ell_0^i$ members of
$\hS{w}_{2i+1}(x)$}.\\
\hS{w}_{2i+2}(x)&=\set{v\notin (\S{b}_{\leq \ell_1}(x)\cup \S{w}_{\leq 2i+1}(x)): v\neq y
\text{ s joined by a black $G_{t_0}$-edge to a vertex in
}\S{w}_{2i+1}(x)}\\
\S{w}_{2i+2}(x)&=\text{the first $\ell_0^i$ members of
$\hS{w}_{2i+2}(x)$}:
\end{align*}
\begin{lemma}\label{lem13}
If $1\leq i\leq \ell_1$, then in $G_{t_0}$, for $c=b,w$,
$$\Pr(|\hS{c}_{i+1}(x)|\leq \ell_0|\S{c}_i(x)|\,\mid\, |\S{c}_j(x)|=\ell_0^{j},\,0\leq j
\leq i)=O(n^{-K})\text{ for any constant }K>0.$$
\end{lemma}
\proofstart
This follows easily from \eqref{notmon} and the Chernoff bounds and $\ell_0^{\ell_1}=o(n)$.
In $G_{n,p_0}$, given that $|\S{c}_{2i}(x)|=\ell_0^i$, each random variable $\hS{c}_{2i+1}(x)$ is binomially distributed with parameters $n-o(n)$ and $1-(1-p_0/2)^{\ell_0^i}$. The mean is therefore asymptotically $\frac12\ell_0^i\log n=\Omega(\log^2n)$ and we are asking for the probability that it is much less than half its mean.
\proofend
It follows from this lemma, that w.h.p., we may define
$\S{b}_0(x),\S{b}_1(x),\ldots,\S{b}_{\ell_1}(x)$ where
$|\S{b}_i(x)|=\ell_0^i$ such that for each $j$ and $z\in \S{b}_j(x)$
there is a zebraic path from $x$ to $z$
that starts with a black edge. For $\S{w}_{\ell_1}(x)$ we can
say the same except that the zebraic
path begins with a white edge.

Having defined the $\S{c}_i(x)$ etc., we define sets
$\S{c}_i(y),i=1,2 \ldots,\ell_1,\,c=b,w$.
We let $\S{b}_0(y)=\S{w}_0(y)=\set{y}$ and then $\S{b}_1(y)$ (resp. $\S{w}_1(y)$)
is the set consisting of the first $\ell_0$ black
(resp. white) neighbors of $y$ that are not in $\S{b}_{\leq
  \ell_1}(x)\cup \S{w}_{\leq \ell_1}(x)$. We note that for $c=b,w$ we
have that w..h.p. $|\hS{c}_1(y)|\geq L_0-18>\ell_0$. This follows from Lemma
\ref{lem12}(b). We can appply this lemma because w.h.p. $t_0\leq \t_1\leq t_1$. Indeed, suppose that $y$ has ten neighbors $T$ in
$\S{w}_{\leq\ell_1}(x)$. Let $S$ be the set of vertices in the paths
from $T$ to $x$ in $\S{w}_{\leq\ell_1}(x)$. If $|S|=s$ then
$S\cup\set{y}$ contains at least $s+9$ edges. This is because every
neighbour after the first adds an additional $k$ vertices and $k+1$ edges to
the subgraph of $G_{t_0}$ spanned by $S\cup\set{y}$,
for some $k\leq \ell_1$. Now $s+1\leq 10\ell_1+1\leq 7L_1$ and the
$s+9$ edges contradict the condition in the lemma, with $a=7$.

We make a slight change in the definitions of the $\hS{c}_i(y)$ in that we keep these sets disjoint from the $\S{c'}_i(x)$.
Thus we take for example
\begin{multline*}
\hS{w}_{2i+1}(y)=\\
\set{v\notin (\S{w}_{\leq 2i}(y)\cup
\S{b}_{\leq \ell_1}(x)\cup \S{w}_{\leq \ell_1}(x)): v
\text{ is joined by a white $G_{t_0}$-edge to a vertex in }\S{w}_{2i}(y)}.
\end{multline*}
 Then we note that excluding $o(n)$ extra vertices has little effect on the
proof of Lemma \ref{lem13} which remains true with $x$
replaced by $y$. We can then define the $\S{c}_{i}(y)$ by taking the first $\ell_0$ vertices.

Suppose now that we condition on the sets $\S{c}_i(x),\S{c}_i(y)$
for $c=b,w$ and $i=0,1,\ldots,\ell_1$. The edges between
the sets with $c=b$ and $i=\ell_1$ and those with $c=w$
and $i=\ell_1$ are unconditioned.
Let
$$\Lambda=\ell_0^{2\ell_1}=n^{4/3-o(1)}.$$
Then, for example, using \eqref{mon}, (strictly speaking, bounding the probability of monotone events in the context of a hypergeometric distribution by the corresponding probability under a binomial distribution),
\beq{eq13}
\Pr(\not\exists\text{ a black $G_{t_0}$ edge joining }\S{b}_{\ell_1}(x),\S{b}_{\ell_1}(y))
\leq 3\brac{1-\frac{\log n}{(2+o(1))n}}^\Lambda=O(n^{-K}),
\eeq
for any positive constant $K$.
\proofend

Thus w.h.p. there is a zebraic path with both terminal edges black between every
pair of large vertices. A similar argument using $\S{w}_{\ell_1}(x),\S{w}_{\ell_1}(y)$
shows that w.h.p. there is a zebraic path with both terminal edges white between
every pair of large vertices.

If we want a zebraic path with a black edge incident with $x$ and a white edge
incident with $y$ then we argue that there is a black $G_{t_0}$ edge between
$\S{b}_{\ell_1}(x)$ and $\S{w}_{\ell_1-1}(y)$.

We now consider the small vertices. Let $V_\s$ be the set of small vertices that
have a large neighbor in $G_{\t_1}$. The above analysis shows that there is a zebraic
path between $v\in V_\s$ and $w\in V_\s\cup V_\l$, where $V_\l$ is the set of large vertices.
Indeed if $v$ is joined by a black edge to a vertex $w\in V_\l$ then we can continue with a zebraic path that begins with a white edge and we can reach any large vertex and choose the
color of the terminating edge to be either black or white. This is useful when we need
to continue to another vertex in $V_\s$.

We now have to deal with small vertices
that have no large neighbors at time $\t_1$. It follows from Lemma \ref{lem12}(a)
that such vertices have degree one or two in $G_{\t_1}$ and that every vertex at distance two
from such a vertex is large.
\begin{lemma}\label{lem14}
All vertices of degree at most two in $G_{t_0}$ are w.h.p.
at distance greater than 10 in $G_{t_1}$,
\end{lemma}
\proofstart
Simpler than Lemma \ref{lem0}(b). We use \eqref{notmon} and then
$$\Pr(\exists \text{ such a pair of vertices})\leq_b t_1^{1/2}\sum_{k=0}^9
n^k p_1^{k-1}\brac{(1-p_0)^{n-k-1}+(n-k)p_0(1-p_0)^{n-k-2}}^2=o(1).$$
\proofend
Let $Z_i$ be the number of vertices of degree $0\leq i\leq 2$ in $G_{t_0}$
that are adjacent in $G_{\t_1}$ to small vertices that are
themselves only incident to edges of one color. Lemma \ref{lem12}(a) implies that 
\beq{eq15}
Z_2=0\ w.h.p.
\eeq

Now consider the case $i=1$. Here we let $Z_1'$ be the number of
vertices of degree one in $G_{t_0}$
that are adjacent in $G_{t_0}$ to vertices that are
themselves only incident to edges of one color. Note that $Z_1\leq Z_1'$.
Then we have, with the aid of \eqref{binom},
\begin{align}
\E(Z_1')&\leq n\binom{n-1}{1}\frac{\binom{N-n+1}{t_0-1}}{\binom{N}{t_0}}
\sum_{{k=1}}^{n-2}\binom{n-2}{k}
\frac{\binom{N-2n+3}{t_0-1-k}}{\binom{N-n+1}{t_0-1}}2^{-(k-1)}.\label{eq11}\\
&\leq_b n^2\frac{t_0}{N}\bfrac{N-t_0}{N-1}^{n-2}\sum_{{k=1}}^{n-2}\binom{n-2}{k} 2^{-k}\bfrac{t_0-1}{N-n+1}^{k}
\bfrac{N-n-t_0+2}{N-n-k+1}^{n-2-k}\nonumber\\
&\leq_b n\log n \exp\set{-\frac{(n-2)(t_0-1)}{N-1}}\sum_{{k=1}}^{n-2}
\binom{n-2}{k}\bfrac{t_0-1}{2(N-n+1)}^k\bfrac{N-n-t_0+2}{N-n-k+1}^{n-2-k}\nonumber\\
&\leq n\log n \exp\set{-\frac{(n-2)(t_0-1)}{N-1}}\sum_{{k=1}}^{n-2}
\binom{n-2}{k}\bfrac{t_0}{2(N-n)}^k\bfrac{N-n-2t_0/3}{N-n}^{n-2-k}\nonumber\\
&\leq_b  \log^3n \brac{\frac{t_0}{2(N-n)}+\frac{N-n-2t_0/3}{N-n}}^{n-2}\nonumber\\
&\leq  \log^3n \bfrac{N-t_0/6}{N-n}^{n-2}\nonumber\\
&=o(1).\label{eq14}
\end{align}
{\bf Explanation for \eqref{eq11}:} We choose a vertex $v$ of degree one
and its neighbor $w$ in $n\binom{n-1}{1}$ ways.
The probability that $v$ has degree one is $\frac{\binom{N-n+1}{t_0-1}}{\binom{N}{t_0}}$.
We fix the
degree of $w$ to be $k+1$. This now has probability
$\frac{\binom{N-2n+3}{t_0-k-1}}{\binom{N-n+1}{t_0-1}}$. The
final factor $2^{-(k-1)}$ is the probability that $w$ only sees edges of one color.

\ignore{
\medskip
In order to deal with $Z_2'$, we next eliminate the possibility of a vertex of degree two in $G_{t_0}$
being in a triangle of $G_{t_1}$. First, using \eqref{mon}, the expected number of
vertices of degree two in $G_{t_0}$ is at most
$$3n\binom{n-1}{2}p_0^2\brac{1-p_0}^{n-3}=O(\log^4n).$$
So, w.h.p. there are fewer than $\log^5n$.

Using \eqref{notmon}, we see that
the expected number of triangles of $G_{t_0}$ containing a vertex of degree two is at most
$$O(t_0^{1/2})\times O(\log^4n)\times n^3p_0^3(1-p_0)^{n-3}=o(1).$$
So, w.h.p. there are no such triangles.

Then the probability that there is an edge of $G_{t_1}-G_{t_0}$ that
joins the two neighbors of a vertex of degree two in $G_{t_0}$ is at most
$$o(1)+\log^5n\times \frac{t_1-t_0}{N}=o(1).$$
Now we can proceed to estimate $\E(Z_2')$, ignoring the possibility
of such a triangle. In which case,
\begin{align}
&\E(Z_2')\nonumber\\
&\leq_b n\binom{n-1}{2}
\frac{\binom{N-n+1}{t_0-2}}{\binom{N}{t_0}}\sum_{k,l=0}^{n-3}
\binom{n-3}{k}\binom{n-3}{l}
\frac{\binom{N-3n+6}{t_0-2-k-l}}{\binom{N-n+1}{t_0-2}}2^{-k-l}\label{eq15}\\
&\leq n^3 \bfrac{t_0}{N}^2\bfrac{N-t_0}{N-2}^{n-3}\times \nonumber\\
&\gap{1}\sum_{k,l=0}^{n-3}
\binom{n-3}{k}\binom{n-3}{l}\bfrac{t_0-2}{N-n+1}^{k+l}\bfrac{N-n-t_0+3}
{N-n-k-l+1}^{2n-5-k-l}2^{-k-l}
\nonumber\\
&\leq_b \log^4n \brac{\sum_{k=0}^{n-3}\binom{n-3}{k}\bfrac{t_0}{2(N-n)}^{k}
\bfrac{N-t_0}{N-2n}^{n-3-k}}^2\nonumber\\
&\leq \log^4n\brac{1-\frac{t_0-2}{2(N-2n)}}^{2(n-3)}\nonumber\\
&=o(1).\label{eq16}
\end{align}
}

\medskip

Finally, consider $Z_0$. Condition on $G_{t_0}$ and assume that
Properties (c),(d) of Lemma \ref{lem12}
hold. For a given isolated vertex, the first $G_{t_0}$ edge incident with it
will have a random endpoint. It follows immediately
that
\beq{eq17}
\Pr(Z_0>0)\leq o(1)+\log^3n\times  \frac{n^{2/3}}{n}=o(1).
\eeq
Here the $o(1)$ accounts for Properties (c),(d) of Lemma \ref{lem12} and
$\log^3n\times n^{-1/3}$ bounds the expected
number of ``first edges'' that choose small endpoints.

Equations \eqref{eq15}, \eqref{eq11} and \eqref{eq17} show that
$Z_0+Z_1+Z_2=0$ w.h.p. In which case it
will be possible to find zebraic paths starting from small vertices. Indeed, we now know that w.h.p. any small vertex $v$ will be adjacent to a vertex $w$ that is incident with edges of both colors and that any other neighbor of $w$ is large.
\section{Proof of Theorem \ref{th5}}\label{th5proof}
The case $r=2$ is implied by Corollary \ref{cor1}. This follows from Corollary \ref{cor1} and \eqref{tau11}. So we can assume that $r\geq3$.
\subsection{$p\leq (1-\e)p_r$}
For a vertex $v$, let
\begin{align*}
C_v&=\set{i:v\text{ is incident with an edge of color $i$}}.\\
I_v&=\set{i:\set{i,i+1}\subseteq C_v}.\qquad{\text{($r+1=1$ here.)}}
\end{align*}
Let $v$ be \emph{bad} if $I_v=\emptyset$. The existence
of a bad vertex means that there are no $r$-zebraic Hamilton cycles. Let $Z_B$
denote the number of bad vertices. Now if $r$ is odd
and $C_v\subseteq \set{1,3,\ldots,2\rdown{r/2}-1}$
or $r$ is even and $C_v\subseteq \set{1,3,\ldots,r-1}$ then
$I_v=\emptyset$. Hence,
$$\E(Z_B)\geq n\brac{1-\frac{\a_rp}{r}}^{n-1}=n^{\e-o(1)}\to\infty.$$
A straightforward second moment calculation shows that $Z_B\neq 0$ w.h.p. and this proves the
first part of the theorem.
\subsection{$p\geq (1+3\e)p_r$}
Note the replacement of $\e$ by $3\e$ here, for convenience. Note also that $\e$ is assumed to be sufficiently small for some inequalities below to hold.

Write $1-p=(1-p_1)(1-p_2)^2$ where $p_1=(1+\e)p_r$ and $p_2\sim\e p_r$ .  Thus
$G_{n,p}$ is the union of $G_{n,p_1}$ and two independent copies of $G_{n,p_2}$.
If an edge appears more than once in $G_{n,p}$, then it retains the color of its first occurence.

Now for a vertex $v$ let $d_i(v)$ denote the number of edges of color $i$ incident
with $v$ in $G_{n,p_1}$. Let
$$J_v=\set{i:d_i(v)\geq \eta_0\log n}$$
where $\eta_0=\e^2/r$.

Let $v$ be \emph{poor} if $|J_v|<\b_r$ where
$\b_r=\rdown{r/2}+1$. Observe that $\a_r+\b_r=r+1$. Then let $Z_P$ denote the number
of poor vertices in $G_{n,p_1}$. A simple calcluation shows that w.h.p. the minimum
degree in $G_{n,p_1}$ is at least $L_0$ and that the maximum degree is at most $6\log n$.
Then
\begin{align*}
\Pr(Z_P>0)&\leq o(1)+n\sum_{k=L_0}^{{6\log n}}\binom{n-1}{k}p_1^k(1-p_1)^{n-1-k}
\sum_{l=r-\b_r+1}^r\binom{r}{l}\binom{k}{l\eta_0\log n}\brac{1-\frac{l}{r}}^{k-r\eta_0\log n}\\
&\leq o(1)+n
\sum_{k=0}^{{6\log n}}\binom{n-1}{k}p_1^k(1-p_1)^{n-1-k}2^r\binom{6\log n}{r\eta_0\log n}\bfrac{\b_r-1}{r}^{k}
\bfrac{r}{\b_r-1}^{r\eta_0\log n}\\
&=o(1)+n2^r\binom{6\log n}{r\eta_0\log n}\bfrac{r}{\b_r-1}^{r\eta_0\log n}\sum_{k=0}^{{6\log n}} (1-p_1)^{n-1}\binom{n-1}{k}\bfrac{p_1(\b_r-1)}{r(1-p_1)}^k\\
&\leq o(1)+2^rn^{1+r\eta_0\log(6e/\eta_0)}(1-p_1)^{n-1}\brac{1+\frac{(\b_r-1)p_1}{r(1-p_1)}}^{n-1}\\
&\leq o(1)+2^rn^{1+r\eta_0\log(6e/\eta_0)}\brac{1-\frac{{(1+o(1))}\a_rp_1}{r}}^{n-1}\\
&=o(1).
\end{align*}
We can therefore assert that w.h.p. there are no poor vertices. This means that
\beq{Kv}
K_v=\set{i:d_i(v),{d_{i-1}(v)}\geq \eta_0\log n}\neq \emptyset\text{ for all }v\in [n].
\eeq
The proof now follows our general 3-phase procedure of (i) finding an $r$-zebraic
2-factor, (ii) removing small cycles so that we have a 2-factor in which
every cycle has length $\Omega(n/\log n)$ and then (iii) using a second
moment calculation to show that this 2-factor can be re-arranged into
an $r$-zebraic Hamilton cycle.

\subsubsection{Finding an $r$-zebraic 2-factor}
We partition $[n]$ into $r$ sets $V_i=[(i-1)n/r+1,in]$ of size $in/r$. Now for each $i$ and each vertex $v$ let
$${N_i(v)=\set{w:\set{v,w}\text{ is an edge of }G_{n,p_1}\text{ of color }i}}.$$
$$d_i^+(v)=|V_{i+1}\cap N_i(v)|\text{ and }d_i^-(v)=|V_{i-1}\cap N_{i-1}(v)|.$$
(Here $r+1$ is interpreted as 1 and 1-1 is interpreted as $r$).

We now let a vertex $v\in V_i$ be $i$-\emph{large} if $d_i^+(v),d_i^-(v)\geq \eta\log n$
where $\eta=\min\set{\eta_0,\eta_1,\eta_2}$ and $\eta_1$ is the solution to
$$\eta_1\log\bfrac{e(1+\e)}{r\eta_1\a_r}=\frac{1}{r\a_r}$$
and $\eta_2$ is the solution to
$$\eta_2\log\bfrac{3er(1+\e)}{\eta_2\a_r}=\frac{1}{3\a_r}.$$
Let $v$ be \emph{large} if it is $i$-large for all $i$.
Let $v$ be \emph{small} otherwise. (Note that $d_i^+(v),d_i^-(v)$ are defined for all $v$, not just for $v\in V_i$, $i\in[r]$).

Let $V_\l,V_\s$ denote the sets of large and small vertices respectively.
\begin{lemma}\label{lem15}
W.h.p., in $G_{n,p_1}$,
\begin{description}
\item[(a)] $|V_\s|\leq n^{1-\th}$ where $\th=\frac{\e}{2r\a_r}$.
\item[(b)] No connected subset of size at most $2\log\log n$ contains more
than $\m_0=r\a_r$ members of $V_\s$.
\item[(c)] If $S\subseteq [n]$ and $|S|\leq n_0=n/\log^2n$ then
$e(S)\leq 100|S|$.
\end{description}
\end{lemma}
\proofstart\\
(a) If $v\in V_\s$ then there exists $i$ such that $d_i^+(v)\leq \eta\log n$
or $d_i^-(v)\leq \eta\log n$. So we have
\begin{align}
\E(|V_\s|)&\leq 2rn\sum_{k=0}^{\eta\log n}\binom{n/r}{k}\bfrac{p_1}{r}^k
\brac{1-\frac{p_1}{r}}^{n/r-k}\label{MK}\\
&\leq3r\bfrac{(1+\e)e}{r\eta\a_r}^{\eta\log n}n^{1-(1+\e+o(1))/r\a_r}\label{35}\\
&\leq n^{1-2\th+o(1)}.
\end{align}
Part (a) follows from the Markov inequality. Note that we can lose the factor 2 in \eqref{MK} since $d^+_i(v)=d^-_{i+2}(v)$.

(b) 
The expected number of connected sets $S$ of size at most $2\log\log n$ containing
$\m_0$ members of $V_\s$ can be bounded by
\beq{206}
\sum_{s=\m_0}^{2\log\log n}\binom{n}{s}s^{s-2}p_1^{s-1}
\binom{s}{\m_0}
\brac{r\sum_{k=0}^{\eta\log n}\binom{n/r-s}{k}\bfrac{p_1}{r}^k
\brac{1-\frac{p_1}{r}}^{n/r-s-k}}^{\m_0}.
\eeq
{\bf Explanation:} We choose $s$ vertices for $S$ and a tree to connect
up the vertices of $S$. We then choose $\m_0$ members $A\subseteq S$
to be in $V_\s$. We multiply by the probability that
for each vertex in $A$, there is at least one $j$ such that $v$ has few
neighbors in $V_{j}\setminus S$ connected to $v$ by edges of color $j$.

After bounding the the sum in brackets raised to $\m_0$ as in \eqref{35}, the sum in \eqref{206} can be bounded by
$$n\sum_{s=\m_0}^{2\log\log n}(4e\log n)^s n^{-\m_0(1+\e+o(1))/r\a_r}=o(1).$$
(c) This is proved in the same manner as Lemma \ref{lem0}(c).
\proofend
For $v\in V_\s$ we let $\f(v)=\min \set{i:i\in K_v}$. Equation \eqref{Kv} implies that $\f(v)$ exists for all $v\in[n]$. Then let $X_i=\set{v\in V_\s:\f(v)=i}$ for $i\in[r]$ and 
$$Y_i=\set{w\notin V_\s:\exists v\in V_\s,s.t.\ (\f(v)=i-1,w\in N_{i-1}(v))\text{ or }(\f(v)=i+1,w\in N_{i}(v))}.$$
It is possible that a vertex $w$ lies in more than one $Y_i$. In which case, delete it from all but one of them. Now let
$$W_i=(V_i\setminus V_\s)\cup X_i\cup Y_i,\quad  i=1,2,\ldots,r.$$
Suppose that $w_i=|W_i|-n/r$ for $i\in [r]$ and let $w_i^+=\max\set{0,w_i}$ for $i\in [r]$. We now remove $w_i^+$ randomly chosen large vertices from each $W_i$ and then randomly assign
$w_i^-=-\min\set{0,w_i}$ of them to each $W_i,i\in [r]$. Thus we obtain a partition of $[n]$ into $r$ sets $Z_i,i=1,2,\ldots,r$, of size $n/r$ for $i\in[r]$.

Let $H_i$ be the bipartite graph induced by $Z_i,Z_{i+1}$ and the edges of color $i$ in $G_{n,p_1}$. We now argue that
\begin{lemma}\label{lemH}
$H_i$ has minimum degree at least $\frac12\eta\log n$ w.h.p.
\end{lemma}
\proofstart
It follows from Lemma \ref{lem15}(b),(d) that no vertex in $Z_i\cap V_i$ loses more than $\m_0$ neighbors from the deletion of $V_\s$ or from the movement of the vertices in the $Y_i$'s. Also, we move $v\in V_\s$ to a $Z_i$ where it has degree at least $\eta\log n-\m_0$ in $V_{i-1}$ and $V_{i+1}$. Its neighborhood may have been affected by the deletion of $V_\s$ or the movement of the $Y_i$'s, but only by at most $\m_0$. Thus for every $i$ and $v\in X_i$, $v$ has at least $\eta\log n-\m_0$ neighbors in $Z_{i-1}$ connected to $v$ by an edge of color $i-1$ and at least $\eta\log n-\m_0$ neighbors in $Z_{i+1}$ connected to $v$ by an edge of color $i$

Now consider the random re-shuffling to get sets of size $n/r$. Fix a $v\in V_i$. Suppose that
it has $d=\Theta(\log n)$ neighbors in $Z_{i+1}$ connected by an edge of color {$i$}. Now randomly choose $w_{i+1}^+=O(|V_\s|\log n)$ vertices to delete from $Z_{i+1}$. The number $\n_v$ of neighbors of $v$ chosen is dominated by $\Bin\brac{{w_{i+1}^+,\frac{d}{n/r}}}$. This follows from the fact that if we choose these ${w_{i+1}^+}$ vertices one by one, then at each step, the chance that the chosen vertex is a neighbor of $v$ is bounded from above by ${\frac{d}{n/r}}$.
So, given the condition in Lemma \ref{lem15}(a) we have
$$\Pr(\n_v\geq 2/\th)\leq \binom{n^{1-\th+o(1)}}{2/\th}\bfrac{dr}{{n}}^{2/\th}
\leq \bfrac{n^{1-\th+o(1)}edr\th}{n}^{2/\th}=o(n^{-1}).$$
\proofend
We can now verify the existence of perfect matchings w.h.p.
\begin{lemma}\label{lem16}
W.h.p., each $H_i$ contains a perfect matching $M_i,i=1,2,\ldots,r$.
\end{lemma}
\proofstart
Fix $i$.
We use Hall's theorem and consider the existence of a set $S\subseteq {Z_i}$
that has fewer than $|S|$ $H_i$-neighbors in ${Z_{i+1}}$. Let $s=|S|$
and let $T=N_{H_i}(S)$ and $t=|T|<s$. We can
rule out $s\leq n_0=n/{2}\log^2n$ through Lemma \ref{lem15}(c). This is because we have
$e(S\cup T)/|S\cup T|\geq \frac14\eta\log n$ in this case.
Let $n_\s=|V_\s|$ and now consider $n/{2}\log^2n\leq s\leq n/2r$.
Given such a pair $S,T$ we deduce that there exist $S_1\subseteq S\subseteq V_i,
|S_1|\geq s-n_\s$ and $T_1\subseteq T\subseteq V_{i+1}$ and
$U_1\subseteq V_{i+1},|U_1|\leq n_\s$
such that there are at least $m_s=(s\eta/2-6n_\s)\log n$ edges between $S_1$ and
$T_1$ and no edges between $S_1$ and $V_{i+1}\setminus(T_1\cup U_1)$.
There is no loss of generality in increasing the size of $T$ to $s$.
We can then write
\begin{align*}
\Pr(\exists\ S,T{\text{ in } G_{n,p_1}})&\leq \sum_{s=n_0}^{n/2r}\binom{n/r-{O(n_\s\log n)}}{s}^2 \binom{s^2}{m_s}p_1^{m_s}(1-p_1)^{(s-n_\s)(n/r-s-n_\s)}\\
&\leq \sum_{s=n_0}^{n/2r}\bfrac{ne}{rs}^{2s}\bfrac{s^2p_1e}{m_s}^{m_s}
e^{-(s-n_\s)(n/r-s-n_\s)p_1}\\
&\leq {\sum_{s=n_0}^{n/2r}}\brac{\bfrac{s}{n}^{\eta\log n/3}\bfrac{3er(1+\e)}{\a_r\eta}^{\eta\log n/2}
n^{-(1-o(1))/2\a_r}}^s\\
&=o(1).
\end{align*}
For the case $s\geq n/2r$ we look for subsets of {$Z_{i+1}$ with too few neighbors in $Z_i$}.
\proofend

It follows from symmetry considerations that the $M_i$ are independent of each other. {Indeed, once we  condition on the number of edges $m_i$ being colored $i=1,2,\ldots,r$, we find that the actual graphs induced by each color are independent of each other. What we have proved implies that for almost all sequences $m_1,m_2,\ldots,m_r$, each $H_i$ has a perfect matching.}
 
Analogously to Lemma \ref{lem2}, we have
\begin{lemma}\label{lem17}
The following hold w.h.p.:
\begin{description}
\item[(a)] $\bigcup_{i=1}^rM_i$ has at most $10\log n$ components.
(Components are $r$-zebraic cycles of length divisible by $r$.)
\item[(b)] There are at most $n_b$ vertices on
components
of size at most $n_c$.
\end{description}
\end{lemma}
\proofstart
The matchings induce a permutation $\p$ on $W_1$. Suppose that $x\in W_1$.
We follow a path via a matching edge to $W_2$ and then by a matching edge to
$W_3$ and so on until we return to a vertex $\p(x)\in W_1$. $\p$ can be taken
to be a random permutation and then the lemma follows from Lemma \ref{lem2}.
\proofend

The remaining part of the proof is similar
to that described in Sections \ref{ics}, \ref{CHC}.
We use the edges of the first copy $G_{n,p_2}$
of color 1 to make all cycles have length $\Omega(n/\log n)$
and then we use the edges of the second copy
of $G_{n,p_2}$ of color 1 to create an $r$-zebraic
Hamilton cycle. The details are left to the reader.
\section{Dealing with the directed analogs}\label{diranalog}
A great deal of the analysis we have seen extends without much comment to the directed case. In particular, in Theorem \ref{th1}, once we have a shown the existences of a matching $M_1$ that is independent of $M_0$, orientation hardly affects the proof. So for Theorem \ref{th1d} all we really need to argue for is a perfect matching $M_1=\set{g_1,g_2,\ldots,g_{n/2}}$ such that if $g_i=\set{x_i,y_i}$ then we can assume that (i) $x_i$ is odd and $y_i$ is even and (ii) $g_i$ is oriented from $y_i$ to $x_i$. For this we will apply Hall's theorem to the bipartite graph $H$ with bipartition $A=\set{2,4,\ldots,n},\,B=\set{1,3,\ldots,n-1}$. $H$ has an edge $\set{a,b}$ iff $(a,b)$ is an edge of $D_m$. The stopping time $\vec{\t}_1$ is for $H$ to have minimum degree one and w.h.p. this will be enough for $H$ to have a perfect matching. After this the proof continues more or less as in the proof of Theorem \ref{th1}. The ``zebraic'' corollary to Theorem \ref{th1d} is not so simple. If we follow the undirected argument then we see that we need to exert control over the orientations of the black and white perfect matchings, they have to be compatible in some sense, and the hitting time for this is not so obvious.

The proof of Theorem \ref{th5d} is almost identical to that of Theorem \ref{th5}. We simply change $I_v$ in Section \ref{th5proof} to
$$I_v=\set{i:d_-^{(i)}>0\text{ and }d_+^{(i+1)}>0},$$
where $d_-^{(i)}$ is the number of edges of color $i$ oriented into $v$ and $d_+^{(i+1)}$ is the number of edges of color $i+1$ oriented out of $v$.

The proof of Theorem \ref{th4d} follows that of Theorem \ref{th4}.

\end{document}